\documentclass{aic}
\aicAUTHORdetails{%
  title = {Planar Graphs in Blowups of Fans},
  author = {Marc Distel, Vida Dujmović, Gwenaël Joret, Piotr Micek, Pat Morin, and David R. Wood},
  plaintextauthor = {Marc Distel, Vida Dujmović, Gwenaël Joret, Piotr Micek, Pat Morin, David R. Wood},
  runningauthor = {M. Distel, V. Dujmović, G. Joret, P. Micek, P. Morin, and D. R. Wood},
  copyrightauthor  = {M. Distel, V. Dujmović, G. Joret, P. Micek, P. Morin, and D. R. Wood},
  keywords = {planar graph, graph minor, product, blowup, bandwidth, sparsification, volume-preserving embedding},
}

\aicEDITORdetails{%
   year={2026},
   number={1},
   received={9 April 2025},   
   published={27 February 2026},  
   doi={10.19086/aic.2026.1},      
}   

\usepackage{amsthm}
\usepackage{amsfonts}
\usepackage{thmtools}
\usepackage{thm-restate}
\usepackage{paralist}
\usepackage[T1]{fontenc}
\usepackage[utf8]{inputenc}
\usepackage{paralist}
\usepackage{halloweenmath}
\usepackage{booktabs}
\usepackage{multirow}
\usepackage{comment}
\usepackage{thm-restate}
\usepackage{mathtools,etoolbox}
\usepackage[nameinlink,noabbrev,capitalise]{cleveref}
\declaretheorem[name=Theorem]{thm}
\crefname{thm}{Theorem}{Theorems}
\declaretheorem[name=Lemma,sibling=thm]{lem}
\declaretheorem[name=Corollary,sibling=thm]{cor}

\declaretheorem[name=Observation,sibling=thm]{obs}
\crefname{obs}{Observation}{Observations}

\declaretheorem[name=Claim,sibling=thm]{clm}
\crefname{lem}{Lemma}{Lemmas}
\crefname{thm}{Theorem}{Theorems}
\crefname{cor}{Corollary}{Corollaries}
\crefname{prop}{Proposition}{Propositions}
\crefname{claim}{Claim}{Claims}
\crefname{conj}{Conjecture}{Conjectures}
\crefname{openproblem}{Open Problem}{Open Problems}
\crefformat{equation}{(#2#1#3)}
\Crefformat{equation}{Equation #2(#1)#3}

\theoremstyle{definition}
\newtheorem{rem}[thm]{Remark}
\crefname{rem}{Remark}{Remark}
\theoremstyle{plain}

\newcommand{\NN}{\mathbb{N}}
\newcommand{\GG}{\mathcal{G}}
\newcommand{\N}{\mathbb{N}}
\newcommand{\R}{\mathbb{R}}
\newcommand{\Z}{\mathbb{Z}}

\newcommand{\ceil}[1]{{\lceil #1 \rceil}}

\newcommand{\floor}[1]{{\lfloor #1 \rfloor}}

\renewcommand{\ge}{\geqslant}
\renewcommand{\le}{\leqslant}
\renewcommand{\geq}{\geqslant}
\renewcommand{\leq}{\leqslant}

\DeclareMathOperator{\crr}{cr}

\usepackage[longnamesfirst,numbers,sort&compress]{natbib}
\definecolor{brightmaroon}{rgb}{0.76, 0.13, 0.28}
\definecolor{linkblue}{rgb}{0, 0.337, 0.227}
\newcommand{\defin}[1]{\emph{\textcolor{brightmaroon}{#1}}}
\makeatletter
\def\mathcolor#1#{\@mathcolor{#1}}
\def\@mathcolor#1#2#3{%
  \protect\leavevmode
  \begingroup
    \color#1{#2}#3%
  \endgroup
}
\makeatother
\newcommand{\mathdefin}[1]{\mathcolor{brightmaroon}{#1}}

\DeclareMathOperator{\tw}{tw}
\DeclareMathOperator{\pw}{pw}
\DeclareMathOperator{\bw}{bw}
\DeclareMathOperator{\td}{td}
\DeclareMathOperator{\rtw}{rtw}
\DeclareMathOperator{\diam}{diam}
\DeclareMathOperator{\mindist}{min-dist}
\DeclareMathOperator{\dist}{dist}
\DeclareMathOperator{\ld}{ld}
\DeclareMathOperator{\polylog}{polylog}
\DeclareMathOperator{\evol}{Evol}
\DeclareMathOperator{\ivol}{Ivol}
\DeclareMathOperator{\tvol}{Tvol}

\newcommand{\scr}[1]{\mathcal{#1}}

\newcommand{\UWT}[2]{#1\langle #2\rangle}

\newcommand{\Apex}[2]{#1^{+#2}}

\begin{document}

\begin{frontmatter}[classification=text]

\title{Planar Graphs in Blowups of Fans}

\author[marc]{Marc Distel\thanks{Research supported by an Australian Government Research Training Program Scholarship.}}
\author[vida]{Vida Dujmović\thanks{Research supported by NSERC, Australian Research Council and a University of Ottawa Research Chair. }}
\author[gwen]{Gwenaël Joret\thanks{Research supported by the Belgian National Fund for Scientific Research (FNRS).}}
\author[piotrek]{Piotr Micek\thanks{Research supported by the National Science Center of Poland under grant UMO-2023/05/Y/ST6/00079 within the WEAVE-UNISONO program.}}
\author[pat]{Pat Morin\thanks{Research supported by NSERC.}}
\author[david]{David R. Wood\thanks{Research supported by the Australian  Research Council and NSERC.}}

\begin{abstract}
  We show that every $n$-vertex planar graph is contained in the graph obtained from a fan by blowing up each vertex by a complete graph of order $O(\sqrt{n}\log^2 n)$.  Equivalently, every $n$-vertex planar graph $G$ has a set $X$ of $O(\sqrt{n}\log^2 n)$ vertices such that $G-X$ has bandwidth $O(\sqrt{n}\log^2 n)$. We in fact prove the same result for any proper minor-closed class, and we prove more general results that explore the trade-off between $X$ and the bandwidth of $G-X$. The proofs use three key ingredients.  The first is a new local sparsification lemma, which shows that every $n$-vertex planar graph $G$ has a set of $O((n\log n)/\delta)$ vertices whose removal results in a graph with local density at most $\delta$. The second is a generalization of a method of Feige and Rao that relates bandwidth and local density using volume-preserving Euclidean embeddings. The third ingredient is graph products, which are a key tool in the extension to any proper minor-closed class. 
\end{abstract}

\end{frontmatter}

\newpage
\tableofcontents
\newpage
\section{Introduction}

This paper studies the global structure of planar graphs and more general graph classes, through the lens of graph blowups\footnote{We consider finite, simple undirected graphs $G$ with vertex set $V(G)$ and edge set $E(G)$. A graph $H$ is \defin{contained} in a graph $G$ if $H$ is isomorphic to a subgraph of $G$.}\footnote{Let $\mathdefin{\NN}:=\{0,1,2,\dots\}$ and $\mathdefin{\NN_1}:=\{1,2,\dots\}$. Let $\R^+:=\{x\in\R:x>0\}$, and for $z\in\R$ let  
$\R_{\geq z}:=\{x\in\R:x\geq z\}$. We use $\log(x)$ for the base-$2$ logarithm of $x$, and we use $\ln(x)$ for the natural logarithm of $x$. When a logarithm appears inside $O$-notation, we use the convention that $\log(x):=1$ for all $x\le 2$.}. Here, the \defin{$b$-blowup} of a graph $H$ is the graph obtained by replacing each vertex $v$ of $H$ with a complete graph $K_v$ of order $b$ and replacing each edge $vw$ of $H$ with a complete bipartite graph with parts $V(K_v)$ and $V(K_w)$, as illustrated in \cref{5blowup}.

\begin{figure}[ht]
   \centering
   \includegraphics{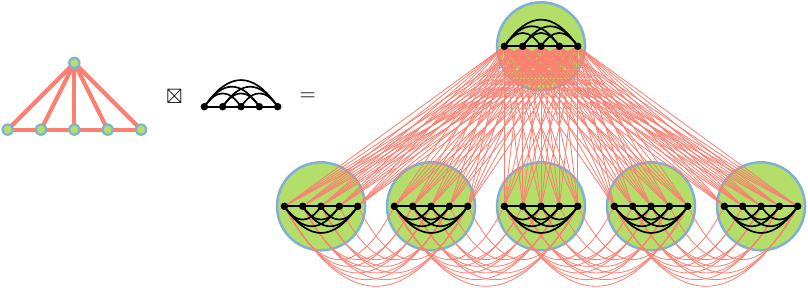}
    \caption{The $5$-blowup of a $6$-vertex fan.}
    \label{5blowup}
\end{figure}

\subsection{Planar Graphs}

Our starting point is the following question: What is the simplest family of graphs $\mathcal{H}$ such that, for each $n$-vertex planar graph $G$, there is a graph $H\in\mathcal{H}$ such that $G$ is contained in a $\tilde{O}(\sqrt{n})$-blowup of $H$, where $\tilde{O}$ notation hides $\polylog(n)$ terms? We show that one can  take $\mathcal{H}$ to be the class of fans, where a \defin{fan} is a graph consisting of a path $P$ plus one \defin{center} vertex adjacent to every vertex in $P$.

\begin{thm}\label{planar-fanblowup}
For any $n\in\NN$ there exists a $O(\sqrt{n}\log^2n)$-blowup of a fan that contains every $n$-vertex planar graph.
\end{thm}

\cref{planar-fanblowup} can be restated in terms of the following classical graph parameter. Let $G$ be a graph. For an ordering  $v_1,\ldots,v_n$ of $V(G)$, let the \defin{bandwidth} $\mathdefin{\bw_G(v_1,\ldots,v_n)}:=\max(\{|j-i|: v_iv_j\in E(G)\}\cup\{0\})$. The \defin{bandwidth} of $G$ is $\mathdefin{\bw(G)}:=\min\{\bw_G(v_1,\ldots,v_n):\text{$v_1,\ldots,v_n$ is an ordering of $V(G)$}\}$. See \citep{CS89,ST20,ABET20,BPTW10,rao:small,BST09,feige:approximating} for a handful of important references on this topic. 
It is well known that bandwidth is closely related to blowups of paths~\citep{DSW07,BPTW10}. Indeed, \cref{planar-fanblowup,planar-bandwidth} are equivalent, where $X$ is the set of vertices mapped to the center of the fan; see \cref{UniversalBlowup,blowup-bandwidth} for a proof. 


\begin{thm}\label{planar-bandwidth}
For any $n\in\NN$, every $n$-vertex planar graph $G$ has a set $X$ of $O(\sqrt{n}\log^2n)$ vertices such that $G-X$ has bandwidth $O(\sqrt{n}\log^2 n)$.
\end{thm}


We in fact prove several generalizations of \cref{planar-bandwidth} that (a) study the trade-off between $|X|$ and $\bw(G-X)$, and (b) consider more general graph classes than planar graphs. 

\subsection{Optimality}
\label{optimality}

Before describing our generalizations of \cref{planar-fanblowup,planar-bandwidth} we show that in various ways these results are best possible, except possibly for the $\log^2 n$ factor.


First note that the blowup factor in \cref{planar-fanblowup} is close to best possible. The $k$-blowup of a graph $H$ with treewidth $t$ has treewidth at most $k(t+1)-1$.\footnote{A \defin{tree-decomposition} of a graph $G$ is a collection $(B_x: x \in V(T))$ of subsets of $V(G)$ indexed by the nodes of a tree $T$, such that: (a) for each edge ${vw \in E(G)}$, there exists a node ${x \in V(T)}$ with ${v,w \in B_x}$, and (b) for each vertex ${v \in V(G)}$, the set $\{ x \in V(T) \colon v \in B_x \}$ induces a non-empty (connected) subtree of $T$. The sets $B_t$ for $t\in V(T)$ are called \defin{bags}. The \defin{width} of such a tree-decomposition is ${\max\{ |B_x| \colon x \in V(T) \}-1}$. A \defin{star} is a tree with at least one vertex, called the \defin{centre}, that is adjacent to every other vertex. A star with at least three vertices has a unique center. If $S$ is a star, then a tree-decomposition $(B_s:s\in V(S))$ is a \defin{star-decomposition}. If $P$ is a path, then a tree-decomposition $(B_p:p\in V(P))$ is a \defin{path-decomposition}, denoted by the corresponding sequence of bags. The \defin{treewidth} of a graph $G$, denoted  $\mathdefin{\tw(G)}$, is the minimum width of a tree-decomposition of $G$. The \defin{pathwidth} of a graph $G$, denoted  $\mathdefin{\pw(G)}$, is the minimum width of a path-decomposition of $G$. Treewidth is the standard measure of how similar a graph is to a tree. Pathwidth is the standard measure of how similar a graph is to a path. By definition, $\tw(G)\leq\pw(G)$ for every graph $G$. 
A \defin{rooted tree} is a tree $T$ with some fixed vertex $r$ called the \defin{root}. The \defin{induced root} of a subtree $T'$ of $T$ is the vertex of $T'$ closest to $r$ in $T$. The \defin{closure} of a rooted tree $T$ is the graph $G$ with $V(G):=V(T)$, where two vertices are adjacent in $G$ if one is an ancestor of the other in $T$. 
The \defin{depth} of a rooted tree $T$ is the maximum number of vertices in a root--leaf path in $T$. The \defin{treedepth} of a graph $G$, denoted $\mathdefin{\td(G)}$, is the minimum depth of  a rooted tree $T$ such that $G$ is contained in the closure of $T$. It is well known and easily seen that $\pw(G)\leq \td(G)-1$. } Since there are $n$-vertex planar graphs of treewidth $\Omega(\sqrt{n})$ (such as the $\sqrt{n}\times\sqrt{n}$ grid), any result like \cref{planar-fanblowup} that finds all planar graphs in blowups of bounded treewidth graphs must have blowups of size $\Omega(\sqrt{n})$
(and fans have treewidth 2, in fact pathwidth 2).  


It follows from the Lipton--Tarjan Planar Separator Theorem that every $n$-vertex planar graph $G$ satisfies $\tw(G)\leq \pw(G)\in O(\sqrt{n})$ (see \citep{Bodlaender98}). It is also well known that  
if $\bw_G(v_1,\dots,v_n)\leq k$ then $\{v_1,\dots,v_{k+1}\},\{v_2,\dots,v_{k+2}\},\dots,\{v_{n-k},\dots,v_n\}$ is a path-decomposition of $G$ of width $k$, implying $\pw(G)\leq \bw(G)$. However, $\bw(G)\in \tilde{O}(\sqrt{n})$ is a much stronger property than $\pw(G)\in \tilde{O}(\sqrt{n})$. Indeed, bandwidth can be $\Omega(n / \log n)$ 
for very simple graphs, such as $n$-vertex complete binary trees. This highlights the strength of \cref{planar-bandwidth}. In fact, \cref{planar-bandwidth} is tight (up to $\polylog$ factors) even for complete binary trees. For a complete binary tree $T$ on $n$ vertices, $\bw(T)\in \Omega(n/\log n)$ since the root of $T$ is within distance $\log n$ of all vertices.  For any set $X\subseteq V(T)$, $T$ contains a complete binary tree with $\Omega(n/|X|)$ vertices that avoids $X$, so $\bw(T-X)\ge\Omega((n/|X|) / \log(n/|X|))$.  Thus, $\bw(T-X)\in\tilde{\Omega}(\sqrt{n})$ for any set $X\subseteq V(T)$ of size $\tilde{O}(\sqrt{n})$.

Pathwidth $2$ is also the best possible bound in results like  \cref{planar-fanblowup}.  Indeed, even \emph{treewidth} $1$ is not achievable:  \citet{LMST08} describe an infinite family of $n$-vertex planar graphs $G$ such that every (improper) 2-colouring has a monochromatic component on $\Omega(n^{2/3})$ vertices. Say $G$ is contained in a $b$-blowup $(K_v:v\in V(T))$ of a tree $T$. Colour each vertex in each $K_v$ by the colour of $v$ in a proper 2-colouring of $T$. Each monochromatic component is contained in some $K_v$, implying that $b\in\Omega(n^{2/3})$.



Any graph of treedepth $c$ has pathwidth at most $c-1$, so it is natural to ask if \cref{planar-fanblowup} can be strengthened to show that every $n$-vertex planar graph is contained in a $\tilde{O}(\sqrt{n})$-blowup of a bounded treedepth graph.  The answer is no, as we now explain. \citet[Theorem~19]{dvowoo} show that, for any $c\in\NN_1$ there exists $\epsilon>0$ such that if the $\sqrt{n}\times\sqrt{n}$ grid is contained in a $b$-blowup of a graph $H$ with treedepth at most $c$, then $b\in\Omega(n^{1/2+\epsilon})$. Thus, the $\sqrt{n}\times \sqrt{n}$-grid is not contained in a $\tilde{O}(n^{1/2})$-blowup of a graph with bounded treedepth.  In particular, \cref{planar-fanblowup} cannot be strengthened to the treedepth setting without increasing the size of the blowup by a polynomial factor.

%

\subsection{Proper Minor-Closed Classes}

We show that \cref{planar-fanblowup} can be generalised for any proper minor-closed graph class\footnote{A \defin{class} is a collection of graphs closed under isomorphism. A class $\GG$ is \defin{monotone} if for every $G\in\GG$, every subgraph of $G$ is in $\GG$. A class $\GG$ is \defin{proper} if some graph is not in $\GG$. A graph $H$ is a \defin{minor} of a graph $G$ if a graph isomorphic to $H$ can be obtained from a subgraph of $G$ by performing any number of contractions; otherwise, $G$ is \defin{$H$-minor-free}. A class $\GG$ is \defin{minor-closed} if for every $G\in\GG$ every minor of $G$ is in $\GG$. A class of graphs $\GG$ is \defin{$H$-minor-free} if every graph in $\GG$ is $H$-minor-free. Note that if $\GG$ is a proper minor-closed class, then $\GG$ is $K_h$-minor-free for some $h\in \NN$.}. 

\begin{thm}
\label{minor-fanblowup}
    For each $h,n\in\NN$ there exists a $O_h(\sqrt{n}\log^2n)$-blowup of a fan that contains every $n$-vertex $K_h$-minor-free graph.
\end{thm}

This result answers a question of \citet{distel.dujmovic.ea:product} up to a $\log^2 n$ factor. They asked if every $K_h$-minor-free graph is contained in a $O_h(\sqrt{n})$-blowup of a graph of bounded pathwidth. Since fans have pathwidth $2$, \cref{minor-fanblowup} gives a positive answer, up to a $\log^2 n$ factor.

\cref{minor-fanblowup} can be rewritten in terms of bandwidth as follows. 

\begin{thm}
\label{minor-bandwidth}
    For each $h,n\in\NN$ every $n$-vertex $K_h$-minor-free graph $G$ has a set $X$ of $O_h(\sqrt{n}\log^2n)$ vertices such that $G-X$ has bandwidth $O_h(\sqrt{n}\log^2 n)$.
\end{thm}

\cref{minor-bandwidth} is applicable for graphs embeddable on any fixed surface. In this setting, we prove a more specialised version of \cref{minor-bandwidth} with much improved dependence on the genus of the surface. In fact, we allow embeddings in surfaces with a bounded number of crossings per edge. These results are presented in \cref{genus_section}.

\subsection{Previous Results}
\label{Previous}



    

\setlength{\tabcolsep}{2.5ex}
\begin{table}[!ht]
\caption{Results on $b$-blowups of a graph $H$ that contain every $n$-vertex graph $G$ from graph class $\mathcal{G}$.}
\label{results_table}
\ \\[.5ex]
\centering
\begin{tabular}{l@{\hspace{1ex}}l@{\hspace{1ex}}l@{\hspace{1ex}}rl@{\hspace{1ex}}r}
\toprule
class $\mathcal{G}$ &$H$
&\multicolumn{2}{c}{lower bound on $b$}
&\multicolumn{2}{c}{upper bound on $b$}\\
  \midrule
\multirow{6}{*}{planar} & tree &
 $\Omega(n^{2/3})$ & \cite{LMST08}
 & $O(n^{2/3})$& \cite{lipton.tarjan:applications}\\[1.5ex]
& $\tw\leq2$
& $\Omega(\sqrt{n})$ &
& $O(\sqrt{n})$ & \cite{distel.dujmovic.ea:product}\\[1.5ex]
& fan
& $\Omega(\sqrt{n})$ &
& $O(\sqrt{n}\log^2 n)$ & \cref{planar-fanblowup}\\[1.5ex]
& $\tw\leq3$
& $\Omega(\tw(G))$ &
& $\tw(G)+1$ &  \cite{ISW24}\\[1ex]
\midrule
     max-degree $\Delta$  & tree
     & $\Omega(\Delta\cdot\tw(G))$ & \cite{Wood09}
     & $O(\Delta\cdot\tw(G))$ & \cite{ding.oporowski:some,Wood09,DW24}\\[1ex]
\midrule
    & $\tw\leq 2$
    & $\Omega(\sqrt{gn})$ & \cite{gilbert.hutchinson.ea:separator}
    & $O((g+1)\sqrt{n})$& \cite{distel.dujmovic.ea:product}\\[1.5ex]
    Euler genus $g$  & $\tw\leq3$
     &  & 
     & $2(g+1)(\tw(G)+1)$ & \cite{ISW24}\\[1.5ex]
     & fan
     & $\Omega(\sqrt{gn})$ & \cite{gilbert.hutchinson.ea:separator}
     & $O(\sqrt{gn}+\sqrt{n}\log^2 n)$ & \cref{genus-fanblowup}\\[1ex]
\midrule
    $(g,k)$-planar & $\tw\leq O(k^3)$
    & & & $O(k^{3/4}g^{3/4} \sqrt{n} )$ & \citep{dvowoo}\\[1.5ex]
    $(g,k)$-planar & $\tw\leq 10^9$
    & $\Omega(\sqrt{gkn})$ & \cite{DEW17}
    & $O_{g,k}(\sqrt{n} )$ & \citep{DHSW24}\\[1.5ex]
    $k$-planar & fan
    & $\Omega(\sqrt{kn})$ & \cite{DEW17}
    & $O(k^{5/4}\sqrt{n}\log^2 n)$ & \cref{k-planar-fanblowup}\\[1.5ex]
    $(g,k)$-planar & fan
    & $\Omega(\sqrt{gkn})$ & \cite{DEW17}
    & $O(g^{1/2} k^{5/4}\sqrt{n}\log^2 n)$ & \cref{gk-planar-fan-blowup}\\[1ex]
\midrule
    $K_{3,t}$-minor-free & $\tw\leq 2$
    & $\Omega(\sqrt{tn})$ && $O(t\sqrt{n})$ & \cite{distel.dujmovic.ea:product} \\[1.5ex]
    $K_h$-minor-free & $\tw\leq h-2$
    & $\Omega(\sqrt{n})$ &
    & $O(\sqrt{hn})$ & \cite{ISW24}\\[1.5ex]
    $K_h$-minor-free & $\tw\leq h-2$
    & $\Omega(\tw(G))$  &
    & $\tw(G)+1$ & \cite{ISW24}\\[1.5ex]
    $K_h$-minor-free & $\tw\leq4$
    & $\Omega(h\sqrt{n})$ & \cite{ast90}
    & $O_h(\sqrt{n})$ & \cite{distel.dujmovic.ea:product}\\[1.5ex]
    $K_h$-minor-free & $\tw\leq3$
    & $\Omega(h\sqrt{n})$ & \cite{ast90}
    & $O_h(\sqrt{n})$ & \cite{Distel25}\\[1.5ex]
    $K_h$-minor-free & fan
    & $\Omega(\sqrt{n})$ &
    & $O_h(\sqrt{n}\log^2 n)$ & \cref{minor-fanblowup}\\[1ex]
\midrule
 & star &
 $\Omega(n^{2/3})$ & \cite{LMST08}
 & $O(n^{2/3})$& \cite{lipton.tarjan:applications}\\[1.5ex]
$O(\sqrt{n})$ treewidth
 & $\td\le O(\log\log n)$
& $\Omega(\sqrt{n})$ & \cite{dvowoo}
& $O(\sqrt{n})$ & \cite{dvowoo}\\[1ex]
    & $\td\leq O(1/\epsilon)$
& $\Omega(n^{1/2+\epsilon})$ & \cite{dvowoo}
& $O(n^{1/2+\epsilon})$ & \cite{dvowoo}\\[1ex]
    \bottomrule
\end{tabular}
\end{table}


As summarized in \cref{results_table}, we now compare the above results with similar results from the literature, starting with the celebrated \defin{Planar Separator Theorem} due to \citet{lipton.tarjan:separator}, which states that any $n$-vertex planar graph $G$ contains a set $X$ of $O(\sqrt{n})$ vertices such that each component of $G-X$ has at most $n/2$ vertices. This theorem quickly leads to results about the blowup structure of planar graphs.  Applying the Planar Separator Theorem recursively shows that any $n$-vertex planar graph $G$ is contained in a graph that can be obtained from the closure of a 
tree of height $O(\log n)$ by blowing up the nodes of depth $i$ into cliques of size $O(\sqrt{n/2^i})$.  (This observation is made by \citet{babai.chung.ea:on} to show that there is a \defin{universal graph} with $O(n^{3/2})$ edges and that contains all $n$-vertex planar graphs.)\ 
By applying it differently, \citet{lipton.tarjan:applications} show that $G$ is contained in a graph obtained from a star by blowing up the center node into a clique of size $n^{1-a}$ and blowing up each leaf into a clique of size $O(n^{2a})$. These two structural results have had an enormous number of applications for algorithms, data structures, and combinatorial results on planar graphs. The second result, with $a=1/3$, shows that $G$ is contained in a $O(n^{2/3})$-blowup of a star. \citet{dvowoo} use the second result recursively (with the size of the separator fixed at $c\sqrt{n}$ even for subproblems of size less than $n$) to show that $G$ is contained in the $O(\sqrt{n})$-blowup of the closure of a tree of height $O(\log\log n)$. That is, $G$ is contained in the $O(\sqrt{n})$-blowup of a graph of treedepth $O(\log\log n)$.  The same method, with the size of the separator fixed at $cn^{1/2+\epsilon}$ shows that $G$ is contained in an $O(n^{1/2+\epsilon})$-blowup of a graph of treedepth $O(1/\epsilon)$ \cite{dvowoo}. See \citep{dvowoo} for more about blowup structure of graph classes that admit $O(n^{1-\epsilon})$-balanced separators.

Using different methods, \citet{ISW24} show that every $n$-vertex planar graph is contained in a $O(\sqrt{n})$-blowup of a graph with treewidth 3. Improving this result, \citet{distel.dujmovic.ea:product} show that every $n$-vertex planar graph is contained in a $O(\sqrt{n})$-blowup of a  treewidth-$2$ graph. They  ask whether every planar graph is contained in a $O(\sqrt{n})$-blowup of a bounded pathwidth graph. Since fans have pathwidth $2$, \cref{planar-fanblowup} answers this question, with $O(\sqrt{n})$ replaced by $O(\sqrt{n}\log^2 n)$.  

Except for the star result (which requires an $\Omega(n^{2/3})$ blowup), all of the above results require blowing up a graph with many high-degree vertices.  \Cref{planar-fanblowup} shows that a pathwidth-$2$ graph with one high-degree vertex is enough, and with a quasi-optimal blowup of $O(\sqrt{n}\log^2 n)$.  Thus, \cref{planar-fanblowup} offers a significantly simpler structural description of planar graphs than previous results.

A related direction of research, introduced by \citet{UTW}, involves showing that every planar graph $G$ is contained in the $b$-blowup of a bounded treewidth graph, where $b$ is a function of the treewidth of $G$. They define the \defin{underlying treewidth} of a graph class $\GG$ to be the minimum integer $k$ such that for some function $f$ every graph $G\in\GG$ is contained in a $f(\tw(G))$-blowup of a graph $H$ with $\tw(H)\leq k$. They show that the underlying treewidth of the class of planar graphs equals 3. In particular, every planar graph $G$ with $\tw(G)\leq t$ is contained in a $O(t^2\log t)$-blowup of a graph with treewidth $3$. \citet{ISW24} reduce the blowup factor to $t+1$. In this setting, treewidth 3 is best possible: \citet{UTW} show that for any function $f$, there are planar graphs $G$ such that if $G$ is contained in a $f(\tw(G))$-blowup of a graph $H$, then $H$ has treewidth at least 3. 

Allowing blowups of size $O(\sqrt{n})$ enables substantially simpler graphs $H$, since \citet{distel.dujmovic.ea:product} show that $\tw(H)\leq 2$ suffices in this $O(\sqrt{n})$-blowup setting. Allowing an extra $O(\log^2n)$ factor in the blowup, the current paper goes further and shows that a fan $H$ (which has pathwidth $2$) suffices. 

For $K_h$-minor-free graphs (which also have treewidth $O_h(\sqrt{n})$~\citep{ast90}), there is a similar distinction between $f(\tw(G))$-blowups and $O_h(\sqrt{n})$-blowups. \citet{UTW} show that the underlying treewidth of the class of $K_h$-minor-free graphs equals $h-2$, whereas \citet{distel.dujmovic.ea:product} show that $\tw(H)\leq 4$ suffices for $O_h(\sqrt{n})$-blowups of $H$. This result was improved to $\tw(H)\leq 3$ by \citet{Distel25}. \cref{minor-fanblowup} improves this result with $H$ a fan, which has pathwidth and treewidth 2, at the expense of an extra $\log^2n$ factor in the blowup. 

\subsection{Bandwidth and Fan-Blowups}

The following straightforward lemma shows that \cref{planar-bandwidth} implies \cref{planar-fanblowup}.

\begin{lem}
\label{UniversalBlowup}
For any $b,n\in\NN_1$, let $\GG$ be the class of $n$-vertex graphs $G$ such that $\bw(G-X)\leq b$ for some $X\subseteq V(G)$ with $|X|\leq b$. 
Let $F$ be the fan on $\ceil{n/b}$ vertices. 
Then the $b$-blowup of $F$ contains every graph in $\GG$.
\end{lem}

\begin{proof}
Let $p:=\ceil{n/b}-1$. Let $F$ be a fan with center $r$, where $F-r$ is the path $u_1,\dots,u_p$. So $|V(F)|=p+1=\ceil{n/b}$. Let $U$ be the $b$-blowup of $F$. Let $G\in\GG$. So $|V(G)|=n$ and $\bw(G-X)\leq b$ for some $X\subseteq V(G)$ with $|X|\leq b$. Move vertices from $G-X$ into $X$ so that $|X|=b$. Still $\bw(G-X)\leq b$. Let $v_1,\dots,v_{n-b}$ be an ordering of $G-X$ with bandwidth at most $b$. Injectively map $X$ to the blowup of $r$. For $i\in\{1,\dots,p-1\}$, injectively map $v_{(i-1)b+1},\dots,v_{ib}$ to the blowup of $u_i$. And injectively map $v_{(p-1)b+1},\dots,v_{n-b}$ to the blowup of $u_p$. By construction, $G$ is contained in $U$. 
\end{proof}

\begin{rem}
    The number of vertices in the fan-blowup $U$ in \cref{UniversalBlowup} is $b\ceil{n/b}$.  When mapping an $n$-vertex graph $G$ to $U$, the $b-(n\bmod b)$ vertices of $U$ not used in the mapping come from the blowup of $u_p$.  By removing these vertices, we obtain a subgraph $U_n$ of $U$ with exactly $n$ vertices that contains every graph in $\mathcal{G}$.  One consequence of this is the following strengthening of \cref{planar-fanblowup}: For each $n\in\N_1$, there exists an \emph{$n$-vertex subgraph} $U_n$ of a $O(\sqrt{n}\log^2 n)$-blowup of a fan that contains every $n$-vertex planar graph.  Each of  \cref{genus-fanblowup,k-planar-fanblowup,gk-planar-fan-blowup,minor-fanblowup} 
    has a similar strengthening.
\end{rem}

The next lemma provides a converse to  \cref{UniversalBlowup}. 

\begin{lem}
\label{blowup-bandwidth}
If an $n$-vertex graph $G$ is contained in a $b$-blowup of a fan $F$, then $\bw(G-X)\leq 2b-1$ for some $X\subseteq V(G)$ with $|X|\leq b$.     
\end{lem}

\begin{proof}
Let $r$ be the center of $F$. Let $X$ be the set of vertices of $G$ mapped to the blowup of $r$. So $|X|\leq b$. By definition, $P:=F-r$ is a path. Let $B_i$ be the set of vertices mapped to the blowup of the $i$-th vertex of $P$. Any ordering of $V(G)$ that places all vertices of $B_i$ before those in $B_{i+1}$ for each $i$ has bandwidth at most $2b-1$. Thus $\bw(G-X)\leq 2b-1$. 
\end{proof}

As mentioned above, we prove generalizations of the above results that explore the trade-off between $|X|$ and $\bw(G-X)$. The following definitions enable this study. For $k\in\R_{\geq 0}$ and $w\in\R_{\geq 1}$, a graph $G$  has \defin{$(k,w)$-bandwidth} if there exists $X\subseteq V(G)$ such that $|X|\leq k$ and $\bw(G-X)\leq w$. A graph $G$ is \defin{$(k,w)$-bandwidth-flexible} if for all $\delta\in \R_{\geq 1}$, $G$ has $(k/\delta,w\delta)$-bandwidth. We call $\delta$ a \defin{multiplier}. If $\delta \geq |V(G)|$, then $G$ trivially has $(k/\delta,w\delta)$-bandwidth with $X=\emptyset$. So we implicitly assume that $1\leq \delta < |V(G)|$ throughout the paper.


A function $f:X\mapsto Y$ with $X,Y\subseteq \R_{\geq 0}$ is \defin{non-decreasing} if $f(n)\leq f(m)$ for all $n,m\in X$ with $n\leq m$, 
and is \defin{superadditive} if $f(n)+f(m)\leq f(n+m)$ for all $n,m\in X$. For a constant $c\in \R^+$, we use the notation \defin{$cf$} to denote the function $n\mapsto cf(n)$, the notation \defin{$f/c$} to denote the function $n\mapsto f(n)/c$, and the notation \defin{$f+c$} to denote the function $n\mapsto f(n)+c$.

For functions 
$f:\NN\mapsto \R_{\geq 0}$ and
$g:\NN\mapsto \R_{\geq 1}$, we say that a class $\GG$ has \defin{$(f,g)$-bandwidth} if each $n$-vertex graph in $\GG$ has $(f(n),g(n))$-bandwidth. And $\GG$ is \defin{$(f,g)$-bandwidth-flexible} if for every $\delta \in \R_{\geq 1}$, $\GG$ has $(f/\delta,\delta g)$-bandwidth. By the above observation, it suffices to assume that $1\leq\delta \leq |V(G)|$ for each graph $G\in\GG$. 


We prove the following results, which with multiplier $\sqrt{n}/\log n$, imply 
\cref{minor-bandwidth,planar-bandwidth}, and with  \cref{UniversalBlowup} imply 
\cref{planar-fanblowup,minor-fanblowup}.

\begin{thm}
\label{planar-flex}
The class of planar graphs is $(f,g)$-bandwidth-flexible for some $f\in O(n\log n)$ and $g\in O(\log^3n)$.
\end{thm}

\begin{thm}
\label{minor-flex}
For any $h\in\NN$, the class of $K_h$-minor-free graphs is $(f,g)$-bandwidth-flexible, where $f\in O_h(n\log n)$ and $g\in O_h(\log^3 n)$.
\end{thm}

From now on, we work primarily in the setting of bandwidth-flexibility, which implies and strengthens the above-mentioned results for fan-blowups. Moreover, bandwidth-flexibility is essential for the proofs; see the application of \cref{starDecomp} in the proof of \cref{mainLem}.

\section{Proof Techniques}

Graph products are a key tool throughout the paper. 
As illustrated in   \cref{strong_product_fig}, the \defin{strong product} $A\boxtimes B$ of two graphs $A$ and $B$ is the graph with vertex set $V(A\boxtimes B):=V(A)\times V(B)$ that contains an edge with endpoints $(v_1,v_2)$  and $(w_1,w_2)$ if and only if
\begin{compactenum}
    \item $v_1w_1\in E(A)$ and $v_2=w_2$;
    \item $v_1=w_1$ and $v_2w_2\in E(B)$; or
    \item $v_1w_1\in E(A)$ and $v_2w_2\in E(B)$.
\end{compactenum}

\begin{figure}[!b]
  \centering
  \includegraphics[page=1]{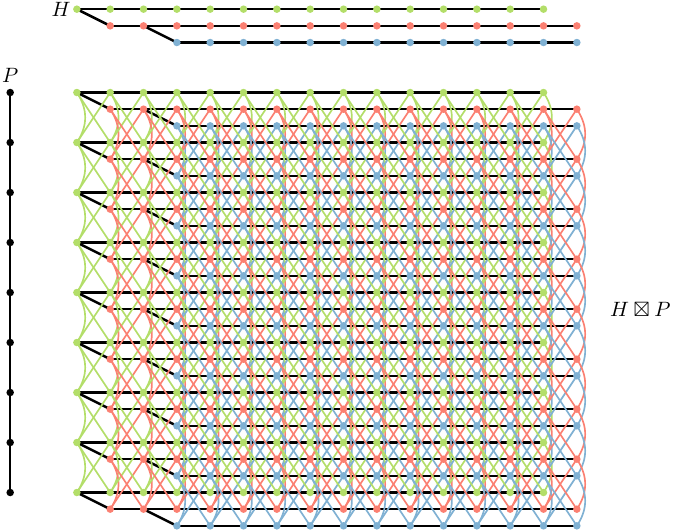}
  \caption{The strong product of a tree $H$ and a path $P$.}
  \label{strong_product_fig}
\end{figure}

Note that the $k$-blowup of $H$ can be written as the strong product $H\boxtimes K_k$.  For example, \cref{planar-fanblowup} states that for every $n$-vertex planar graph $G$ there is a fan $F$ such that $G$ is isomorphic to a subgraph of $F\boxtimes K_{O(\sqrt{n}\log^2 n)}$.

Throughout this paper we work with the product of a bounded treewidth graph and a path. The \defin{row treewidth} of a graph $G$, denoted by \defin{$\rtw(G)$}, is the minimum integer $t$ such that $G$ is contained in $H\boxtimes P$ for some graph $H$ with treewidth $t$ and for some path $P$. Note that row treewidth does not increase under taking subgraphs or adding isolated vertices. The \defin{row treewidth} of a class $\GG$, \defin{$\rtw(\GG)$}, is the maximum row treewidth of a graph in $\GG$, or $\infty$ if the maximum does not exist. This definition is motivated by the result of \citet{dujmovic.joret.ea:planar} who proved that planar graphs have bounded row treewidth. In particular, they showed that $\rtw(G)\leq 8$ for every planar graph $G$, improved to $\rtw(G)\leq 6$ by \citet{ueckerdt.wood.ea:improved}. Thus \cref{planar-flex} for planar graphs is an immediate consequence of the following more general result. 



\begin{thm}
\label{rtw-flex}
Every $n$-vertex graph $G$ with row treewidth at most $b$ is $(f,g)$-bandwidth-flexible for some $f\in O(bn\log n)$ and $g\in O(\log^3n)$.
\end{thm}



Graph products are a key tool in our proof of \cref{minor-fanblowup} for $K_h$-minor-free graphs. In particular, we use the following `Graph Minor Product Structure Theorem' of \citet{dujmovic.joret.ea:planar}, which gives a structural description of $K_h$-minor free graphs in terms of products (replacing surfaces, vortices and weak apex vertices in the `Graph Minor Structure Theorem' of \citet{RS-XVI})\footnote{The adhesion bound in \cref{GMST} is not stated in \citep{dujmovic.joret.ea:planar}, but it is implied since every graph with row treewidth $b$ has clique-number at most $2(b+1)$, implying that $\UWT{G}{B_t}$ has clique-number at most $2(b+1)+\max(h-5,0)$.}. The \defin{adhesion} of a tree-decomposition $(B_t:t\in V(T))$ of a graph $G$ is $\max_{tt'\in E(T)}|B_t\cap B_{t'}|$. The \defin{torso} of $G$ at a vertex $t\in V(T)$ (with respect to $(B_t:t\in V(T))$), denoted \defin{$\UWT{G}{B_t}$}, is the graph obtained from $G[B_t]$ by adding an edge $uv$ whenever there exists $tt'\in E(T)$ such that $u,v\in B_t\cap B_{t'}$, provided $uv$ does not already exist in $G[B_t]$.

\begin{thm}[\citep{dujmovic.joret.ea:planar}]
\label{GMST}
For each $h\in\NN_1$, there exist $b,k\in\NN_1$ such that every $K_h$-minor-free graph $G$ admits a tree-decomposition $(B_t:t\in V(T))$ of adhesion at most $k$ such that for each $t\in V(T)$, there exists $X_t\subseteq B_t$ such that $|X_t|\leq \max(h-5,0)$ and $\UWT{G}{B_t}-X_t$ has row treewidth at most $b$. 
\end{thm}

We now sketch the main ideas in the proof of \cref{rtw-flex}.


For a graph $G$ and any two vertices $v,w\in V(G)$, define the \defin{(graph) distance} between $v$ and $w$, denoted $\mathdefin{d_G(v,w)}$, as the minimum number of edges in any path in $G$ with endpoints $v$ and $w$ or define $d_G(v,w):=\infty$ if $v$ and $w$ are in different components of $G$.  For any $r\ge 0$ and any $v\in V(G)$, let $\mathdefin{B_G(v,r)}:=\{w\in V(G):d_G(v,w)\le r\}$ denote the radius-$r$ ball in $G$ with center $v$.
The \defin{local density} of a graph $G$ is $\mathdefin{\ld(G)}:=\max\{(|B(v,r)|-1)/r:r> 0,\, v\in V(G)\}$.\footnote{The ${}-1$ in this definition of local density does not appear in the definitions of local density used in some other works \cite{feige:approximating,rao:small}, but this makes no difference to our asymptotics results.  Our definition makes for cleaner formulas and seems to be more natural. For example, under our definition, the local density of a cycle of length $2k+1$ is $2$ and every $r$-ball contains exactly $2r+1$ vertices for $r\in\{1,\ldots,k\}$. Without the ${}-1$, the local density of a cycle is $3$, but only because radius-$1$ balls contain three vertices.}



The local density of $G$ provides a lower bound on the bandwidth of $G$. For any ordering $v_1,\dots,v_n$ of $V(G)$ with bandwidth $b$, for each vertex $v_i$, $B_G(v_i,r) \subseteq \{v_{i-rb},\dots,v_{i+rb}\}$, so $|B_G(v_i,r)|\leq 2rb+1$ and $\ld(G)\leq 2\bw(G)$. In 1973, Erd\H{o}s conjectured that $\bw(G)\leq O(\ld(G))$ for every graph $G$ \cite[Section~3]{ccdg82}. This was disproved by \citet{chvatalova:on} who describes a family of $n$-vertex trees $T$ with $\ld(T)\le 25/3$ and $\bw(T)\in\Omega(\log n)$.\footnote{The proof of Theorem~3.4 in \cite{chvatalova:on} constructs an infinite tree with vertex set $\N^2$ that has local density at most $25/3$ and infinite bandwidth.  In this construction, for each $h\in\N$, the maximal subtree that includes $(0,a_h)$ but not $(0,a_{h}+1)$ has $n_h\le 2\cdot 8^{h}$ vertices and bandwidth at least $h/9\in\Omega(\log n_h)$.} Thus, $\bw(G)$ is not upper bounded by any function of $\ld(G)$, even for trees. This remains true for trees of bounded pathwidth: \citet{CS89} describe a family of $n$-vertex trees $T$ with local density at most $9$, pathwidth $2$, and  $\bw(T)\in\Omega(\log n/\log\log n)$. On the other hand, in his seminal work, \citet{feige:approximating} proves that bandwidth is upper bounded by the local density times a polylogarithmic function of the number of vertices.


\begin{thm}[\citet{feige:approximating}]
\label{feige_bandwidth_vs_density}
For any $n\in\NN$, for every $n$-vertex graph $G$,
  \[
    \bw(G)\in O\left(\ld(G)\cdot \log^3 n\sqrt{\log n\log\log n}\right) \enspace .
  \]
\end{thm}

\citet{rao:small} improves \cref{feige_bandwidth_vs_density} in the special case of planar graphs:


\begin{thm}[\citet{rao:small}]
\label{rao_bandwidth_vs_density}
For any $n\in\NN$, for every $n$-vertex planar graph $G$, 
  \[
    \bw(G)\in O\left(\ld(G)\cdot \log^3 n\right) \enspace .
  \]
\end{thm}

By \cref{rao_bandwidth_vs_density}, to prove \cref{planar-fanblowup} it suffices to show the following \defin{local sparsification lemma}:


\begin{lem}\label{planar_sparsifier}
  For any $D \in\R_{\geq1}$ and $n\in\NN$, every $n$-vertex planar graph $G$ has a set $X$ of $O((n\log n)/D )$ vertices such that $G-X$ has local density at most $D$.
\end{lem}

\cref{planar_sparsifier} and the bandwidth upper bound in 
\cref{rao_bandwidth_vs_density} for graphs of given local density, together show that the class of planar graphs is $(O(n\log n),O(\log^3 n))$-bandwidth-flexible, which establishes \cref{planar-flex}. 
\cref{planar_sparsifier} is proved in \cref{local_sparsification_section}.





Proving \cref{rtw-flex} is the subject of \cref{htimesp_section} and is the most technically demanding aspect of our work, for reasons that we now explain. \Cref{rao_bandwidth_vs_density} is not stated explicitly in \cite{rao:small}.  It is a consequence of the following two results of \citet{feige:approximating} and \citet{rao:small}. (The definition of $(k,\eta)$-volume-preserving contractions is in \cref{contractions}, but is not needed for the  discussion that follows):

\begin{thm}[\cite{rao:small}]\label{rao_planar_graphs}
For all $k,n\in\NN_{1}$, every $n$-vertex planar graph has a $(k,O(\sqrt{\log n}))$-volume-preserving Euclidean contraction.
\end{thm}

\begin{thm}[\cite{feige:approximating}]\label{feige_bandwidth_vs_density_graphs}
For any $n\in\NN$, every $n$-vertex graph $G$ with local density $\ld(G)\leq D$ that has a $(k,\eta)$-volume-preserving Euclidean contraction,\footnote{The precise trade-off between all these parameters is not stated explicitly in \cite{feige:approximating}, but can be uncovered from Feige's proof, which considers the case where $k=\log n$ and $\eta=\sqrt{\log n}\sqrt{\log n+ k\log k}$.}
    \[
        \bw(G) \in O((nk\log n)^{1/k}\,Dk\eta\log^{3/2} n) \enspace .
    \]
\end{thm}

\cref{rao_bandwidth_vs_density} is an immediate consequence of \cref{rao_planar_graphs,feige_bandwidth_vs_density_graphs} with $k=\ceil{\log n}$.  Unfortunately, we are unable to replace ``planar graph'' in \cref{rao_planar_graphs} with ``subgraph of $H\boxtimes P$.''  The proof of \cref{rao_planar_graphs} relies critically on the fact that planar graphs are $K_{3,3}$-minor-free.  Specifically, it uses the Klein--Plotkin--Rao (KPR) decomposition \cite{klein.plotkin.ea:excluded} of $K_{h}$-minor-free graphs $G$, which partitions $V(G)$ into parts so that the diameter of each part $C$ in $G$ is $\diam_G(C)\in O_h(\Delta)$ (for $O(\log n)$ different values of $\Delta$).\footnote{The diameter of a subset $S\subseteq V(G)$ in $G$ is $\mathdefin{\diam_{G}(S)}:=\max\{d_G(v,w):v,w\in S\}$. In recent work on coarse graph theory~(e.g.~\citep{DN23,BBEGLPS}), $\diam_G(S)$ is called the `weak diameter' of $S$, to distinguish it from the diameter of $G[S]$.}  
This does not help because $H\boxtimes P$ is not $K_h$-minor-free for any fixed $h$, even when $H$ is a path.



Although $H\boxtimes P$ is not necessarily $K_h$-minor-free, a very simple (two-step) variant of the KPR decomposition accomplishes some of what we want.  That is, it provides a partition of $V(G)$ so that each part $C$ has $\diam_{H\boxtimes P}(C)\in O(\Delta)$. However, distances in $G$ can be much larger than distances in $H\boxtimes P$, so this decomposition does not provide upper bounds on $\diam_G(C)$.  To deal with this, we work with distances in $H\boxtimes P$, so that we can use the simple variant of the KPR decomposition.

Working with distances in $H\boxtimes P$ requires that we construct a set $X$ of vertices so that the metric space $\mathcal{M}:=(V(G)\setminus X,d_{(H\boxtimes P)-X})$ has local density $O(\sqrt{tn}/\log n)$.  That is, we must find a set $X$  of vertices in $H\boxtimes P$ so that radius-$r$ balls in the graph $(H\boxtimes P)-X$ contain at most $rD+1$ vertices of $G-X$, for $D=\sqrt{tn}/\log n$.  As it happens, the same method used to prove \cref{planar_sparsifier} (the local sparsification lemma for planar graphs) provides such a set $X$.

However, we are still not done.  The simple variant of the KPR decomposition guarantees bounds on $\diam_{H\boxtimes P}(C)$, but does not guarantee bounds on $\diam_{(H\boxtimes P)-X}(C)$, which is what we now need.  This is especially problematic because $G-X$ may contain pairs of vertices $v$ and $w$ where $d_{(H\boxtimes P)-X}(v,w)$ is unnecessarily much larger than $d_{H\boxtimes P}(v,w)$.  This happens, for example, when vertices added to $X$ to eliminate overly-dense radius-$r$ balls happen to increase the distance between $v$ and $w$ even though no overly-dense radius-$r$ ball contains $v$ and $w$.

To resolve this problem, we introduce a distance function $d^*$ that mixes distances measured in $H\boxtimes P$ with distance increases intentionally caused by ``obstacles'' in $X$.  This contracts the shortest path metric on $(H\boxtimes P)-X$ just enough so that, for each part $C$ in (a refinement of) the simplified KPR decomposition, $\diam_{d^*}(C)\in O(\Delta)$. The trick is to do this in such a way that $d^*$ does not contract the metric too much, so the local density of the metric space $\mathcal{M}^*:=(V(G)\setminus X,d^*)$ is $O(\sqrt{tn}/\log n)$, just like the metric space $\mathcal{M}$ that it contracts.  At this point, we can follow the steps used in Rao's proof to show that the metric space $\mathcal{M}^*$ has a $(k,O(\sqrt{\log n}))$-volume-preserving Euclidean contraction (the equivalent of \cref{rao_planar_graphs}) and then apply a generalization of \cref{feige_bandwidth_vs_density_graphs} to establish that $G-X$ has bandwidth $O(\sqrt{tn}\log^2 n)$.

\section{Local Sparsification}
\label{local_sparsification_section}

This section proves a generalization of our local sparsification lemma, \cref{planar_sparsifier}. The proof uses the following standard vertex-weighted separator lemma. Similar results with similar proofs appear in  \citet{robertson.seymour:graph}, but we provide a proof for the sake of completeness.

\begin{lem}\label{weighted_separator}
    Let $H$ be a graph; let $\mathcal{T}:=(B_x:x\in V(T))$ be a tree-decomposition of $H$; and let $\xi:V(H)\to\R_{\geq 0}$ be a function.  For any subgraph $X$ of $H$, let  $\xi(X):=\sum_{v\in V(X)} \xi(v)$.    Then, for any $c\in\NN_1$, there exists $S\subseteq V(T)$ such that $|S|\le c-1$ and for each component $X$ of $H-(\bigcup_{x\in S} B_x)$, $\xi(X)\le \xi(H)/c$.
\end{lem}

\begin{proof}
  The proof is by induction $c$. The base case $c=1$ is trivial, since $S:=\emptyset$ satisfies the requirements of the lemma.  Now assume $c\ge 2$.  Root $T$ at some arbitrary vertex $r$ and for each $x\in V(T)$, let $T_x$ denote the subtree of $T$ induced by $x$ and all its descendants.  Let $H_x:=H[\bigcup_{y\in V(T_x)} B_y]$.  Say that a node $x$ of $T$ is \defin{heavy} if $\xi(H_x) \ge \xi(H)/c$. Since $c\ge 1$, $r$ is heavy, so $T$ contains at least one heavy vertex. Let $y$ be a heavy vertex of $T$ with the property that no child of $y$ is also heavy.  Then $H':=H-V(H_y)$ has weight $\xi(H') = \xi(H)-\xi(H_y) \le (1-1/c)\cdot\xi(H)$.  On the other hand, every component $C$ of $H-V(H')-B_y$ has weight $\xi(C) \le \xi(H)/c$.  Apply induction on the graph $H'$ with tree-decomposition $\mathcal{T}':=(B_x\cap V(H'):x\in V(T))$ and $c':=c-1$ to obtain a set $S'$ of size at most $c-2$ such that each component $X$ of $H'-(\bigcup_{x\in S'} B_x)$, has weight $\xi(X) \le \tfrac{1}{c-1}\cdot(1-\tfrac{1}{c})\cdot\xi(H) = \tfrac{1}{c}\cdot \xi(H)$.  The set $S:=S'\cup\{y\}$ satisfies the requirements of the lemma.
\end{proof}

A \defin{layering} $\{L_s:s\in\Z\}$ of a graph $G$ is a collection of pairwise disjoint sets indexed by the integers whose union is $V(G)$ and such that, for each edge $vw$ of $G$, $v\in L_i$ and $w\in L_j$ implies that $|i-j|\le 1$.
For example, if $r$ is a vertex in a connected graph $G$, and $L_i:=\{v\in V(G):d_G(v,r)=i\}$ for each integer $i\geq \NN$, then $\{L_i:i\in\NN\}$ is a layering of $G$, called a \defin{BFS layering}. For $t\in\NN_1$, a layering $\{L_s:s\in\Z\}$ of a graph $G$ is \defin{$t$-Baker} if, for every $s\in\Z$ and $r\in\NN_1$, $G[L_s\cup\cdots\cup L_{s+r-1}]$ has treewidth at most $rt-1$. A graph $G$ is \defin{$t$-Baker} if $G$ has a $t$-Baker layering.  Clearly, if every connected component of $G$ is $t$-Baker, then $G$ is $t$-Baker.

Every planar graph is $3$-Baker, and for a connected planar graph $G$, any BFS layering of $G$ is $3$-Baker~\cite{RS-III}. (This property is used in Baker's seminal work on approximation algorithms for planar graphs~\cite{baker:approximation}.) Thus, \cref{planar_sparsifier} is an immediate consequence of the following more general result:


\begin{lem}
\label{sparsifier_baker}
For any $D \in\R_{\geq 1}$ and $t,n\in\NN_1$, any $n$-vertex $t$-Baker graph $G$ contains a set $X$ of at most $(18tn\log n)/D$ vertices such that $G-X$ has local density at most $D$.
\end{lem}

\begin{proof}
If $D \geq n$ then the claim holds with $X=\emptyset$. Now assume that $D <n$. Let $\mathcal{L}:=\{L_s:s\in \Z\}$ be a $t$-Baker layering of $G$.  Without loss of generality, assume that $L_i=\emptyset$ for each $i <0$ and each $i\ge n$.  For each $i\in\NN$ and $j\in\Z$, let $G_{i,j}:=G[\bigcup_{s=j2^i}^{(j+1)2^i-1} L_s]$,  and let $G^+_{i,j}=G[V(G_{i,j-1})\cup V(G_{i,j})\cup V(G_{i,j+1})]$.  Observe that, for every $i$, the graphs in $\{ G_{i,j}\}_{j\in\N}$ are pairwise vertex disjoint.  By the definition of $G^+_{i,j}$, this implies that the graphs in $\{G^+_{i,j}\}_{j\in\N}$ have a total of at most $3n$ vertices.

  For each $i\in\{0,\ldots,\lfloor \log n\rfloor-1\}$ 
  and each $j$, $G^+_{i,j}$ has treewidth at most $3t\cdot 2^i-1$, since $\mathcal{L}$ is $t$-Baker.  By \cref{weighted_separator}, with weight function $\xi(v)\coloneqq 1$ for every $v\in V(G^+_{i,j})$ and $c \coloneqq \lceil |V(G^+_{i,j})|/(D 2^{i-1})\rceil$, there exists a set $X_{i,j}\subseteq V(G^+_{i,j})$ such that 
  \[
    |X_{i,j}|\le 3t\cdot 2^i\cdot(c-1) =
    3t\cdot 2^i\cdot\left(\left\lceil\frac{ |V(G^+_{i,j})|}{D 2^{i-1}}\right\rceil-1\right) \le
    \frac{3t\cdot 2^i\cdot|V(G^+_{i,j})|}{ D 2^{i-1}}
    = \frac{6t|V(G^+_{i,j})|}{D},
  \]
  and each component of $G^+_{i,j}-X_{i,j}$ has at most $|V(G^+_{i,j})|/c \le \delta2^{i-1}$ vertices.  Let
  \[X:=\bigcup_{i=0}^{\lfloor\log n\rfloor-1}\bigcup_{j}X_{i,j}.\]
  Then
  \[
    |X| \le
    \sum_{i=0}^{\lfloor\log n\rfloor-1}\sum_{j} |X_{i,j}|
    \le
    \sum_{i=0}^{\lfloor\log n\rfloor-1}\sum_{j} \frac{6t|V(G^+_{i,j})|}{D}
    \le
    \sum_{i=0}^{\lfloor\log n\rfloor-1}\frac{18tn}{D}
    \le\frac{18tn\log n}{D}
    \enspace .
  \]
  Now, consider some ball $B_{G-X}(v,r)$ in $G-X$, let $i=\lceil\log r\rceil$, and let $j$ be the unique integer such that $v\in V(G_{i,j})$.  Then $B_{G-X}(v,r)$ is contained in a single component of $G^+_{i,j}-X_{i,j}$, and this component has at most $D 2^{i-1}=D 2^{\lceil\log r\rceil-1}\le D  r$ vertices.
\end{proof}

\section{Volume-Preserving Contractions}
\label{contractions}

This section introduces volume-preserving Euclidean contractions, and explains their connection to bandwidth. This material builds on the work of \citet{feige:approximating} and \citet{rao:small}, and is essential for the proof of \cref{rtw-flex} in \cref{htimesp_section}.


A \defin{distance function} over a set $S$ is any function $d:S^2\to\R\cup\{\infty\}$ that satisfies $d(x,x)=0$ for all $x\in S$; $d(x,y)=d(y,x)>0$ for all distinct $x,y\in S$; and $d(x,z) \le d(x,y)+d(y,z)$ for all distinct $x,y,z\in S$.  For any $x\in S$, and any non-empty $Z\subseteq S$, $\mathdefin{d(x,Z)}:=\min(\{d(x,y):y\in Z\})$.  A \defin{metric space} $\mathcal{M}:=(S,d)$ consists of a set $S$ and a distance function $d$ over (some superset of) $S$.
$\mathcal{M}$ is \defin{finite} if $S$ is finite and $\mathcal{M}$ is \defin{non-empty} if $S$ is non-empty.  For $x\in S$ and $r\geq 0$, the \defin{$r$-ball} centered at $x$ is $\mathdefin{B_{\mathcal{M}}(x,r)}:=\{y\in S:d(x,y)\le r\}$.  The \defin{diameter} of a non-empty finite metric space $(S,d)$ is $\mathdefin{\diam_d(S)}:=\max\{d(x,y):x,y\in S\}$, and the
\defin{minimum-distance} of $(S,d)$ is $\mathdefin{\mindist_d(S)}:=\min(\{d(x,y): \{x,y\}\in\binom{S}{2}\}\cup\{\infty\})$.
Note that $\mindist_d(S)=\infty$ if $|S|=1$. 

For any graph $G$, $d_G$ is a distance function over $V(G)$, so $\mathdefin{\mathcal{M}_G}:=(V(G),d_G)$ is a metric space. Any metric space that can be defined this way is referred to as a \defin{graph metric}. For any $S\subseteq V(G)$, the \defin{diameter} and \defin{minimum-distance} of $S$ in $G$ are defined as $\mathdefin{\diam_G(S)}:=\diam_{d_G}(S)$ and $\mathdefin{\mindist_G(S)}:=\mindist_{d_G}(S)$, respectively.

Since we work with strong products it is worth noting that, for any two graphs $A$ and $B$,
\[
  d_{A\boxtimes B}((x_1,x_2),(y_1,y_2))=\max\{d_A(x_1,y_1),d_B(x_2,y_2)\} \enspace .
\]

Define the \defin{local density} of a non-empty finite metric space  $\mathcal{M}=(S,d)$ to be
\[
  \mathdefin{\ld(\mathcal{M})}:=
  \max\{\,(|B_{\mathcal{M}}(x,r)|-1)/r\,:\,x\in S,\, r>0\}. 
\]
(This maximum exists because $S$ is finite, so there are only $\binom{|S|}{2}$ values of $r$ that need to be considered.)\ 
Thus, if $\mathcal{M}$ has local density at most $D$,  then $|B_{\mathcal{M}}(x,r)|\le Dr+1$ for each $x\in S$ and $r\ge 0$.
This definition is consistent with the definition of local density of graphs:  A graph $G$ has local density at most $D$ if and only if the metric space $\mathcal{M}_G$ has local density at most $D$.  Note that, if $(S,d)$ has local density at most $D$ then $(S,d)$ has $\diam_d(S)\ge (|S|-1)/D$ and $\mindist(S)\ge 1/D$.

A \defin{contraction} of a metric space $\mathcal{M}=(S,d)$ into a metric space $\mathcal{M'}=(S',d')$ is a function $\phi:S\to S'$ that satisfies $d'(\phi(x),\phi(y))\le d(x,y)$, for each $x,y\in S$. The \defin{distortion} of $\phi$ is $\max\{d(x,y)/d'(\phi(x),\phi(y)):\{x,y\}\in \binom{S}{2}\}$.\footnote{If there exists $\{x,y\}\in \binom{S}{2}$ with $d(x,y)>0$ and $d'(\phi(x),\phi(y))=0$, then the distortion of $\phi$ is infinite. This is not the case for any of the contractions considered in this work.}  When $S\subseteq S'$ and $\phi$ is the identity function, we say that $\mathcal{M'}$ is a contraction of $\mathcal{M}$. In particular, saying that $(S,d')$ is a contraction of $(S,d)$ is equivalent to saying that $d'(x,y)\le d(x,y)$ for all $x,y\in S$.

For two points $x,y\in\R^L$, let $\mathdefin{d_2(x,y)}$ denote the Euclidean distance between $x$ and $y$.  A contraction of $(S,d)$ into $(\R^L, d_2)$ for some $L\ge 1$ is called a \defin{Euclidean contraction}.  For $K\subseteq S$ we abuse notation slightly with the shorthand $\mathdefin{\phi(K)}:=\{\phi(x):x\in K\}$.   We make use of two easy observations that follow quickly from these definitions:

\begin{obs}\label{contraction_increases_density}
  Let $\mathcal{M}:=(S,d)$ and $\mathcal{M}':=(S',d')$ be non-empty finite metric spaces.  If $\mathcal{M'}$ has local density $D$ and $\mathcal{M}$ has an injective contraction into $\mathcal{M}'$ then  $\mathcal{M}$ has local density at most $D$.
\end{obs}

\begin{proof}
  Let $\phi:S\to S'$ be an injective contraction of $\mathcal{M}$ into $\mathcal{M}'$.  For every $x\in S$, every $r > 0$, and every $y\in B_\mathcal{M}(x,r)$, we have $d'(\phi(x),\phi(y))\le d(x,y)\le r$, since $\phi$ is a contraction.  Therefore, $B_{\mathcal{M'}}(\phi(x),r)\supseteq \phi(B_{\mathcal{M}}(x,r))$.  Since $\phi$ is injective, $|B_{\mathcal{M'}}(\phi(x),r)|\ge |\phi(B_{\mathcal{M}}(x,r))|=|B_{\mathcal{M}}(x,r)|$.
  Since $\mathcal{M}'$ has local density at most $D$, $rD+1\ge |B_{\mathcal{M'}}(\phi(x),r)|\ge  |B_{\mathcal{M}}(x,r))|$.
\end{proof}

\begin{obs}\label{supergraph_contraction}
  For any graph $I$ and any subgraph $G$ of $I$, $(V(G),d_I)$ is a contraction of $(V(G),d_G)$.
\end{obs}

\begin{proof}
  From the definitions, it follows that $d_I$, restricted to $V(G)$ is a distance function over $V(G)$, so $(V(G),d_I)$ is a metric space.  Since $G$ is a subgraph of $I$, every path in $G$ is also a path in $I$ so, $d_I(x,y)\le d_G(x,y)$ for each $x,y\in V(G)$.  
\end{proof}



For a set $K$ of $k\le L+1$ linearly-independent points in $\R^L$, the \defin{Euclidean volume} of $K$, denoted by $\mathdefin{\evol(K)}$, is the $(k-1)$-dimensional volume of the simplex whose vertices are the points in $K$.  For example, if $k=3$, then $\evol(K)$ is the area of the triangle whose vertices are $K$ and that is contained in a plane that contains $K$.

Define the \defin{ideal volume} of a finite metric space $(K,d)$ to be 
\[\mathdefin{\ivol_d(K)}:=\max\{\evol(\phi(K)):\text{$\phi$ is a Euclidean contraction of $(K,d)$}\} \enspace .\]
A Euclidean contraction $\phi:S\to\R^{\ell}$ of a finite metric space $(S,d)$ is \defin{$(k,\eta)$-volume-preserving} if $\evol(\phi(K))\ge \ivol_d(K)/\eta^{k-1}$ for each $k$-element subset $K$ of $S$.  This definition is a generalization of distortion: $\phi$ is $(2,\eta)$-volume-preserving if and only if $\phi$ has distortion at most $\eta$.

\citet{feige:approximating} introduces the following definition and theorem as a bridge between ideal volume and Euclidean volume. The \defin{tree volume} of a finite metric space $(K,d)$ is defined as $\mathdefin{\tvol_d(K)}:=\prod_{xy\in E(T)} d(x,y)$ where $T$ is a minimum spanning tree of the weighted complete graph with vertex set $K$ where the weight of each edge $xy$ is equal to $d(x,y)$.  The following lemma makes tree volume a useful intermediate measure when trying to establish that a contraction is volume-preserving.

\begin{lem}[{\citet[Theorem~3]{feige:approximating}}]\label{tvol_vs_ivol}
For any finite metric space  $(S,d)$ with $|S|=k$, 
  \[
    \ivol_{d}(S) \le \frac{\tvol_d(S)}{(k-1)!} \le 2^{(k-2)/2}\ivol_d(S) \enspace .
  \]
\end{lem}


The following lemma, whose proof appears in \cref{reciprocal_sum_section}, generalizes \citet[Theorem~10]{feige:approximating} from graph metrics to general metric spaces and establishes a critical connection between local density and tree volume.

\begin{restatable}[{Generalization of \cite[Theorem~10]{feige:approximating}}]{lem}{reciprocalsum}\label{reciprocal_sum}
For any $k,n\in\NN_1$, for every $n$-element metric space $\mathcal{M}:=(S,d)$ with local density at most $D$,
  \[
    \sum_{K\in \binom{S}{k}}\frac{1}{\tvol_{d}(K)} < n(DH_n/2)^{k-1} \enspace ,
  \]
  where $\mathdefin{H_n}:=\sum_{i=1}^n 1/i\le 1+\ln n$ is the \defin{$n$-th harmonic number}.
\end{restatable}

\Cref{volume_density_bandwidth}, which appears below and whose proof appears in
\cref{volume_density_bandwidth_section}, is a generalization of \cref{feige_bandwidth_vs_density_graphs} from graph metrics to arbitrary metrics.
 %
 %
 %
 %
First, we need a definition of bandwidth for metric spaces.  Let $(S,d)$ be a non-empty finite metric space and let $x_1,\ldots,x_n$ be a permutation of $S$.  Then $\mathdefin{\bw_{(S,d)}(x_1,\ldots,x_n)}:=\max\{j-i:d(x_i,x_j)\le 1,\, 1\le i <j\le n\}$ and $\mathdefin{\bw(S,d)}$ is the minimum of $\bw_{(S,d)}(x_1,\ldots,x_n)$ taken over all $n!$ permutations $x_1,\ldots,x_n$ of $S$.  Note that this coincides with the definition of the bandwidth of a graph: For any connected graph $G$, $\bw(\mathcal{M}_G)=\bw(G)$.  First, observe that injective contractions can only increase bandwidth:



\begin{obs}\label{contraction_increases_bandwidth}
  For every finite metric space $\mathcal{M}:=(S,d)$ and every (injective) contraction $\mathcal{M}':=(S,d')$ of $\mathcal{M}$, $\bw(\mathcal{M}) \le \bw(\mathcal{M}')$.
\end{obs}

\begin{proof}
  Let $x_1,\ldots,x_n$ be an ordering of the elements of $S$ such that $b:=\bw(\mathcal{M}')=\linebreak \bw_{\mathcal{M}'}(x_1,\ldots,x_n)$. Consider any pair of elements $x_ix_j$ with $d(x_i,x_j) \le 1$. Since $\mathcal{M}'$ is a contraction of $\mathcal{M}$, $d'(x_i,x_j)\le 1$.  Since $\bw_{\mathcal{M}'}(x_1,\ldots,x_n)\le b$, $|j-i|\le b$.  Thus $\bw(\mathcal{M})\le \bw_{\mathcal{M}}(x_1,\ldots,x_n)\le b$.
\end{proof}

\begin{restatable}[{Generalization of \cref{feige_bandwidth_vs_density_graphs}}]{thm}{volumedensitybandwidth}
\label{volume_density_bandwidth}
  Let $D>0$ and $\Delta\ge 2$, and let $(S,d)$ be an  $n$-element metric space with local density at most $D$ and diameter at most $\Delta$.  If $(S,d)$ has a $(k,\eta)$-volume-preserving Euclidean contraction $\phi:S\to\R^L$ then
  \[
    \bw(S,d) \in O((nk\log\Delta)^{1/k}\,Dk\eta\log^{3/2} n) \enspace .
  \]
\end{restatable}

%
%
%

\section{\boldmath Subgraphs of \texorpdfstring{$H\boxtimes P$}{H boxtimes P}: Proof of \texorpdfstring{\cref{rtw-flex}}{Theorem?}}
\label{htimesp_section}

This section proves a bandwidth-flexibility result for  graphs of given row treewidth (\cref{rtw-flex}) that generalizes \cref{planar-flex} for planar graphs, and is an essential ingredient in the proof of the analogous result for $K_h$-minor-free graphs (\cref{minor-flex}). 

Most of the results in this section are written as claims that are not self-contained, since they refer $G$, $H$, $P$, $X$, $d^*$, and other objects defined throughout this section. From this point on, $\mathdefin{G}$ is an $n$-vertex subgraph of $H\boxtimes P$ where $\mathdefin{H}$ is a \defin{$t$-tree} (an edge-maximal graph of treewidth $t$) and $\mathdefin{P}$ is a path.

We now outline the structure of our proof, where $\delta$ is the given multiplier. (We use the notation $\mathcal{M}\rightbroom\mathcal{M'}$ to denote that $\mathcal{M}'$ is a contraction of $\mathcal{M}$.)\ 

\begin{enumerate}
  \item Use a variant of \cref{sparsifier_baker} to find a set $X\subseteq V(H\boxtimes P)$ of size $O((tn\log n)/\delta)$ such that the metric space $\mathcal{M}:=(V(G-X),d_{(H\boxtimes P)-X})$ has local density at most $\delta$. 
  Since $G-X$ is a subgraph of $(H\boxtimes P)-X$, \cref{supergraph_contraction} implies that $\mathcal{M}$ is a contraction of the metric space $\mathcal{M}_{G-X}:=(V(G-X),d_{G-X})$, so $\mathcal{M}_{G-X}\rightbroom\mathcal{M}$.

  \item Design a distance function $d^*:V((H\boxtimes P)-X)^2\to\R$ so that the metric space $\mathcal{M}^*:=(V(H\boxtimes P)\setminus X,d^*)$ 
  is a contraction of $\mathcal{M}$ with the property that the induced metric space $(V(G-X),d^*)$ has local density at most $\delta$.

  Graphically, $\mathcal{M}_{G-X}\rightbroom\mathcal{M}\rightbroom\mathcal{M}^*$.

  \item Prove that $\mathcal{M}^*$ has a $(k,O(\sqrt{\log n}))$-volume-preserving Euclidean contraction, for $k=\ceil{\log n}$.  The preceding two steps are done in such a way that this part of the proof is able to closely follow the proof of \cref{rao_planar_graphs} by \citet{rao:small}.

  \item  By \cref{volume_density_bandwidth},  $\bw(\mathcal{M}^*)\in O(\delta\log^3 n)=O(\sqrt{tn}\log^2 n)$.  Since $\mathcal{M}^*$ is a contraction of $\mathcal{M}_{G-X}$, \cref{contraction_increases_bandwidth} implies that $\bw(G-X)=\bw(\mathcal{M}_{G-X}) \le \bw(\mathcal{M}^*)\in O(\sqrt{tn}\log^2 n)$.
\end{enumerate}

The delicate part of the proof is the design of the distance function $d^*$ that contracts $d_{(H\boxtimes P)-X}$ but still ensures that the local density of $(V(G-X),d^*)$ is at most $\delta$. If $d^*$ contracts too much, then $(V(G-X),d^*)$ will not have local density $O(\delta)$. If $d^*$ contracts too little, then it will be difficult to get a $(k,O(\sqrt{\log n}))$-volume-preserving Euclidean embedding of $\mathcal{M}^*$.  To make all of this work, the distance function $d^*$ makes use of the structure of the sparsifying set $X$.

\subsection{A Structured Sparsifier}
\label{x_definition}


In this section, we construct a sparsifying set $X$ like that used in \cref{planar_sparsifier}.  The main difference is that we do not use a BFS layering of $G$ when applying \cref{sparsifier_baker}. Instead, we use the layering of $G$ that comes from $H\boxtimes P$.  Although this is really the only difference, we repeat most of the steps in the proof of \cref{sparsifier_baker} in order to establish notations and precisely define the structure of $X$, which will be useful in the design of the distance function $d^*$.  In particular, later sections rely on the structure of the individual subsets $X_{i,j}$ whose union is $X$.

Let $N:=2^{\ceil{\log n}}$ and let $P:=\mathdefin{y_{-N+1},y_{-N+2},\ldots,y_{2N}}$ be a path.  Without loss of generality we assume all vertices of $G$ are contained in $V(H)\times\{y_1,\ldots,y_N\}$.  For each $i\in\{0,\ldots,\log N\}$ and each $j\in\{-1,0,\ldots,N/2^{i}\}$, let $P_{i,j}:=y_{j2^i+1},\ldots,y_{(j+1)2^i}$ be a subpath of $P$ with $2^i$ vertices. For each $i\in\{0,\ldots,\log N\}$ and each $j\in\{0,\ldots,N/2^{i}-1\}$, let $P^+_{i,j}:=P[V(P_{i,j-1})\cup V(P_{i,j})\cup V(P_{i,j+1})]$ be the concatenation of $P_{i,j-1}$, $P_{i,j}$, and $P_{i,j+1}$. Define $Q_{i,j}\coloneqq H\boxtimes P_{i,j}$ and $Q^+_{i,j}\coloneqq H\boxtimes P^+_{i,j}$.
In words, $Q_{i,0},\ldots,Q_{i,N/2^i-1}$ partitions the part of $H\boxtimes P$ that contains $G$ into vertex-disjoint strips of height $2^i$. Each subgraph $Q^+_{i,j}$ is a strip of height $3\cdot 2^i$ that contains $Q_{i,j}$ in its middle third.

To construct our sparsifying set $X$, we first construct vertex subsets $Y_{i,j}$ of $H$ for each $i\in\{0,1,\ldots,\log N\}$ and $j\in\{0,\ldots,N/2^{i}-1\}$.  Define the weight function $\xi_{i,j}:V(H)\to\N$ where $\xi_{i,j}(x):=|(\{x\}\times V(P^+_{i,j})) \cap V(G)|$.  Observe that $\xi_{i,j}(H):=\sum_{x\in V(H)} \xi_{i,j}(x)=|V(Q^+_{i,j})\cap V(G)|$. Let $D\ge 2$ be a real number. By \cref{weighted_separator} with $c:=\lceil\xi_{i,j}(H)/(2^{i-1}\delta)\rceil$, there exists $Y_{i,j}\subseteq V(H)$ of size at most $(t+1)\xi_{i,j}(H)/(2^{i-1}\delta)$, such that each component $C$ of $H-Y_{i,j}$ has total weight $\xi_{i,j}(C)\le 2^{i-1}\delta$.  For each $i\in\{0,1,\ldots,\log N\}$ and $j\in\{0,\ldots,N/2^{i}-1\}$, let $\mathdefin{X_{i,j}}:=Y_{i,j}\times V(P^+_{i,j})$.  We think of $X_{i,j}$ as a vertical separator that splits the strip $Q^+_{i,j}$ into parts using vertex cuts that run from the top to the bottom of $Q^+_{i,j}$.

\begin{clm}\label{component_sizes}
  For each $i\in\{0,,\ldots,\log N\}$ and $j\in\{0,\ldots,N/2^{i}-1\}$, each component of $Q^+_{i,j}-X_{i,j}$ has at most $2^{i-1}\delta$ vertices.\
\end{clm}

\begin{proof}
  The number of vertices of $G$ in a component $C$ of $Q^+_{i,j}-X_{i,j}$ is equal to the total weight $\zeta_{i,j}(C_H)$ of the corresponding component $C_H$ of $H-Y_{i,j}$.  Therefore, each component of $Q^+_{i,j}-X_{i,j}$ contains at most $2^{i-1}\delta$ vertices of $G$.
\end{proof}


Let 
\[
    \mathdefin{X} :=
     \bigcup_{i=0}^{\text{log} N} {
     \bigcup_{j=0}^{N/2^i-1}} X_{i,j} \enspace .
\]

\begin{clm}\label{x_size}
  $|X|\le 18(t+1)n(1+\log N)/\delta$.
\end{clm}

\begin{proof} 
  Observe that $\sum_{j=0}^{N/2^i-1} \xi_{i,j}(H)\le\sum_{j=0}^{N/2^i-1} 3|V(Q_{i,j})| \le 3n$, since, each vertex $v$ of $G$ can only appear in $Q^+_{i,j-1}, Q^+_{i,j}$, and $Q^+_{i,j+1}$ where $j$ is the unique index such that $v\in V(Q_{i,j})$.    By definition, $|X_{i,j}|=3\cdot 2^i \cdot |Y_{i,j}| \le 6(t+1)\xi_{i,j}(H)/\delta$.  Therefore, $\sum_{j=0}^{N/2^i-1} |X_{i,j}|\le 18(t+1)n/\delta$. Summing over $i\in\{0,\ldots,\log N\}$ completes the proof.
\end{proof}

\subsection{\boldmath The Distance Function \texorpdfstring{$d^*$}{d*}}
\label{d_star_definition}

In order to construct a volume-preserving Euclidean contraction $\phi$ for a distance function $d$ we must ensure (at least) that $d_2(\phi(v),\phi(w))$ is large whenever $d(v,w)$ is large.  This is relatively easy to do for the distance function $d_{H\boxtimes P}$ using (simplifications of) the techniques used by \citet{rao:small} for planar graphs. This is more difficult for $d_{(H\boxtimes P)-X}$ because distances are larger, which only makes the problem harder.  Some of these distances are necessarily large; the obstacles in $X$ are needed to ensure that $(V(G),d_{(H\boxtimes P)-X})$ has local density at most $D$.  The purpose of a single set $X_{i,j}$ is to increase distances between some pairs of vertices in $Q^+_{i,j}$ so that they are at least $2^i$.  However, the obstacles in $X$ sometimes interact, by chance, to make distances excessively large. \cref{big_distance} shows that, even when $H=P$, obstacles in $X_{i,j+1}$ and in $X_{i,j-1}$ can interact in such a way that $d_{(H\boxtimes P)-X}(v,w)$ can become $r2^i$ for arbitrarily large $r$.  This large distance is not needed to ensure the local density bound and it makes it difficult to construct a volume-preserving Euclidean contraction of $(V(G),d_{(H\boxtimes P)-X})$.  The purpose of the intermediate distance function $d^*$ is to reduce these unnecessarily large distances so that $d^*(v,w)$ is defined only by the ``worst'' obstacle in $X$ that separates $v$ and $w$.


\begin{figure}
  \centering
  \includegraphics{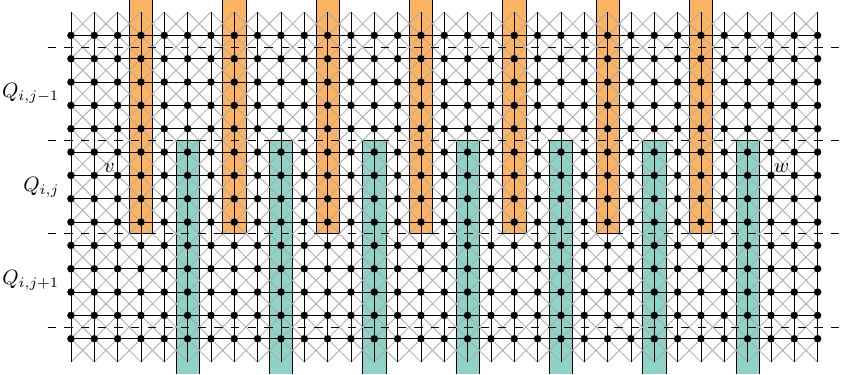}
  \caption{Obstacles not in $X_{i,j}$ can interact to create excessively large distances between vertices in $Q_{i,j}$.}
  \label{big_distance}
\end{figure}

For any subgraph $B$ of a graph $A$, we use the shorthand $\overline{B}:=V(A)\setminus V()$. (When we use this notation, the graph $A$ will be clear from context.)\ 
For any vertex $u$ of $H\boxtimes P$, let $u_P$ denote the second coordinate of $u$ (the projection of $u$ onto $P$).  Let $u$ and $v$ be two vertices of $(H\boxtimes P)-X$.  If $u$ and $v$ are both vertices of $Q^+_{i,j}$ but are in different components of $Q^+_{i,j}-X_{i,j}$, then define
\[
  \mathdefin{d_{i,j}(u,v)}:=\min \left\{d_P(u_P,x) + d_P(x,v_P):x\in \overline{P^+_{i,j}}\right\} \enspace .
\]
Otherwise (if one of $u$ or $v$ is not in $Q^+_{i,j}$ or $u$ and $v$ are in the same component of $Q^+_{i,j}-X_{i,j}$), define $\mathdefin{d_{i,j}(u,v)}:=0$.  When $d_{i,j}(u,v)>0$, it is helpful to think of $d_{i,j}(u,v)$ as the length of the shortest walk in $P$ that begins at $u_P$, leaves $P^+_{i,j}$ and returns to $v_P$.  Now define our distance function
\[
  \mathdefin{d^*(u,v)}:=\max\left(\{d_{H\boxtimes P}(u,v)\}\cup \left\{d_{i,j}(u,v):(i,j)\in\{0,\ldots,\log N\}\times\{1,\ldots,N/2^i-1\}\right\}\right) \enspace .
\]
Intuitively, $d^*(u,v)$ captures the fact that any path from $u$ to $v$ in $(H\boxtimes P)-X$ must navigate around each obstacle $X_{i,j}$ that separates $u$ and $v$ in the graph $Q^+_{i,j}$.  At the very least, this requires a path from $u$ to some vertex $x$ outside of $Q^+_{i,j}$ followed by a path from $x$ to $v$.  The length of this path is at least the length of the shortest walk in $P$ that begins at $u_P$, contains $x_P$ and ends at $w_P$.

\begin{clm}\label{d_star_distance}
  The function $d^*:V((H\boxtimes P)-X)\to\N\cup\{\infty\}$ is a distance function for $V((H\boxtimes P)-X)$.
\end{clm}

\begin{proof}
 It is straightforward to verify that $d^*(u,u)=0$ for all $u\in V((H\boxtimes P)\setminus X)$ and that $d^*(u,v)=d^*(v,u)\ge 0$ for all $u,v\in V((H\boxtimes P)-X)$.  It only remains to verify that $d^*$ satisfies the triangle inequality.  We must show that, for distinct $u,v,w\in V((H\boxtimes P)-X)$, $d^*(u,w)\le d^*(u,v)+d^*(v,w)$.

  If $d^*(u,w)=d_{H\boxtimes P}(u,w)$ then $d^*(u,v)+d^*(v,w)\ge d_{H\boxtimes P}(u,v)+d_{H\boxtimes P}(v,w)\ge d_{H\boxtimes P}(u,w)$ and we are also done.  Otherwise, $d^*(u,w)=d_{i,j}(u,w)>0$ for some $i,j$.  Then $u$ and $w$ are vertices of $Q^+_{i,j}$ that are in different components of $Q^+_{i,j}-X_{i,j}$.  There are two cases to consider, depending on the location of $v$:
  \begin{compactenum}
    \item If $v\not\in V(Q^+_{i,j})$ then $d^*(u,v)+d^*(v,w)\ge d_{H\boxtimes P}(u,v)+d_{H\boxtimes P}(v,w) \ge d_P(u_P,v_P)+d_P(v_P,w_P)\ge d_{i,j}(u,w)=d^*(u,w)$.

    \item If $v\in V(Q^+_{i,j})$ then, since $u$ and $w$ are in different components $C_u$ and $C_w$ of $Q^+_{i,j}-X_{i,j}$, at least one of $C_u$ or $C_w$ does not contain $v$.  Without loss of generality, suppose $C_w$ does not contain $v$.  Then $d^*(u,v)+d^*(v,w)\ge d_{H\boxtimes P}(u,v)+d_{i,j}(v,w) \ge d_{P}(u_P,v_P)+d_{i,j}(v,w)$.  Now, $d_{P}(u_P,v_P)$ is the length of a path in $P$ from $u_P$ to $v_P$ and $d_{i,j}(v,w)$ is the length of a (shortest) walk in $P$ that begins at $v_P$, leaves $P^+_{i,j}$ and then returns to $w_P$. Thus,   $d_{P}(u_P,v_P)+d_{i,j}(v,w)$ is the length of a walk in $P$ that begins at $u_P$, leaves $P^+_{i,j}$ and then returns to $w_P$. On the other hand, $d_{i,j}(u,w)$ is the length of a shortest walk in $P$ that begins at $u_P$, leaves $P^+_{i,j}$ and returns to $w_P$, so $d_{i,j}(u,w)\le d_P(u_P,v_P)+d_{i,j}(v,w)$.  Therefore, $d^*(u,v)+d^*(v,w)\ge d_{P}(u_P,v_P)+d_{i,j}(v,w)\ge d_{i,j}(u,w)=d^*(u,w)$. \qedhere
  \end{compactenum}
\end{proof}

\begin{clm}\label{delta_density}
  The metric space $\mathcal{M}^*:=(V(G-X),d^*)$ has local density at most $D$.
\end{clm}

\begin{proof}
  We must show that, for any $v\in V(G)$ and any $r>0$, $|B_{\mathcal{M}^*}(v,r)|\le Dr+1$.  If $r\ge n/D$ then this is trivial, so assume that $r< n/D$.  Consider some vertex $w\in B_{\mathcal{M}^*}(v,r)$.  Let $i:=\lceil\log r\rceil$ and let $j$ be such that $v$ is a vertex of $Q_{i,j}$.  Since $w\in B_{\mathcal{M}^*}(v,r)$, $d_{H\boxtimes P}(v,w)\le r\le 2^i$.  Therefore $d_P(v_P,w_P)\le d_{H\boxtimes P}(v,w)\le 2^i$.  Therefore $w$ is contained in $Q^+_{i,j}$.  Since $d^*(v,w)\le r$, $d_{i,j}(v,w)\le r$.  This implies that $v$ and $w$ are in the same component of $Q^+_{i,j}-X_{i,j}$ since, otherwise, $d_{i,j}(v,w)\ge d_P(u_P,\overline{P^+_{i,j}}) + d_P(\overline{P^+_{i,j}},w_P)\ge 2^i+1$.
  Therefore, $B_{\mathcal{M}^*}(v,r)$ is contained in the component $C$ of $Q^+_{i,j}-X_{i,j}$ that contains $v$.  By \cref{component_sizes}, $|V(C)|\le 2^{i-1}D< rD$.  Therefore, $|B_{\mathcal{M}^*}(v,r)|\le |V(C)|< rD$.
\end{proof}

\begin{clm}
  The metric space $(V(H\boxtimes P)\setminus X,d^*)$ is a contraction of $\mathcal{M}_{(H\boxtimes P)-X}=(V((H\boxtimes P)-X),d_{(H\boxtimes P)-X})$.
\end{clm}

\begin{proof}
  Let $u$ and $v$ be distinct vertices of $(H\boxtimes P)-X$.  If $d^*(u,v)=d_{H\boxtimes P}(u,v)$ then, $d^*(u,v)=d_{H\boxtimes P}(u,v)\le d_{(H\boxtimes P)-X}(u,v)$.  If $d^*(u,v)=d_{i,j}(u,v)$ for some $i$ and $j$ then any path from $u$ to $v$ in $(H\boxtimes P)-X$ must contain some vertex $x$ not in $Q^+_{i,j}$ since $u$ and $v$ are in different components of $Q^+_{i,j}-X$.  The shortest such path has length at least $d_{H\boxtimes P}(u,x)+d_{H\boxtimes P}(x,v) \ge d_P(u_P,x_P)+d_P(x_P,v_P) \ge d_{i,j}(u,v)=d^*(u,v)$.
\end{proof}

The preceding claims are summarized as follows:

\begin{cor}\label{d_star_summary}
  The metric space $(V((H\boxtimes P)-X), d^*)$ is a  contraction of $(V((H\boxtimes P)-X), d_{(H\boxtimes P)-X})$ and the metric space $(V(G-X),d^*)$ has local density at most $D$.
\end{cor}

\subsection{Volume-Preserving Contraction of \texorpdfstring{$\mathcal{M}^*$}{M*}}


This subsection proves the following result:


\begin{clm}\label{dstar_contraction}
For every integer $k\in\{1,\ldots,n\}$, the metric space $\mathcal{M}^*:=(V(G-X),d^*)$ has a \linebreak $(k,O(\sqrt{\log n}))$-volume-preserving Euclidean contraction.
\end{clm}

\paragraph{\boldmath Decomposing $H\boxtimes P$.}

Let $\Delta\ge 4$ be a power of $2$. We now show how to randomly decompose $H\boxtimes P$ into subgraphs $\{(H\boxtimes P)^\Delta_{a,b}:(a,b)\in\Z^2\}$. The only randomness in this decomposition comes from choosing two independent uniformly random integers $r_H$ and $r_P$ in $\{0,\ldots,\Delta-1\}$. See \cref{chop_fig} for an example.

\begin{figure}
  \centering
  \includegraphics[page=2]{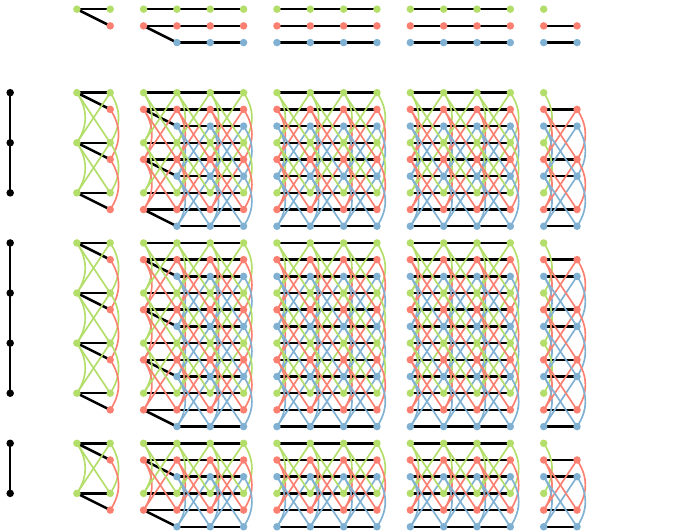}
  \caption{The result of decomposing the graph $H\boxtimes P$ in \cref{strong_product_fig} with $\Delta=4$, $r_H=2$, and $r_P=3$.}
  \label{chop_fig}
\end{figure}

Let $\{L_s:s\in\Z\}$ be a BFS layering of $H$.  For each integer $a$, let $H^\Delta_a:=H[\bigcup_{s=r_H+a\Delta}^{r_H+(a+1)\Delta-1} L_s]$ so that $\{H^\Delta_a:a\in \Z\}$ is a pairwise vertex-disjoint collection of induced subgraphs that covers $V(H)$ and each $H^\Delta_a$ is a subgraph of $H$ induced by $\Delta$ consecutive BFS layers. For each integer $b$, let $P^\Delta_b:=P[\{y_{r_P+b\Delta},\ldots,y_{r_P+(b+1)\Delta-1}]$ so that $\{P^\Delta_b:b\in\Z\}$ is a collection of vertex disjoint paths, each having $\Delta$ vertices, that cover $P$.   For each $(a,b)\in\Z^2$, let $(H\boxtimes P)^\Delta_{a,b}:= H^\Delta_a \boxtimes P^\Delta_b$.


\begin{clm}\label{component_diameter}
  For each $(a,b)\in\Z^2$, each component $C$ of $(H\boxtimes P)^\Delta_{a,b}$ has $\diam_{H\boxtimes P}(C)\le 2\Delta+1$.
\end{clm}

\begin{proof}
  Let $v:=(v_1,v_2)$ and $w:=(w_1,w_2)$ be two vertices of $(H\boxtimes P)^\Delta_{a,b}$.  Our task is to show that $d_{H\boxtimes P}(v,w)\le 2\Delta+1$.  Recall that
  $d_{H\boxtimes P}(v,w)=\max\{d_{H}(v_1,w_1),d_{P}(v_2,w_2)\}$.  Since $P^\Delta_b$ is a subpath of $P$ with $\Delta$ vertices, $d_P(v_2,w_2)=d_{P^\Delta_b}(v_2,w_2)\le\Delta-1$, so we need only upper bound $d_{H}(v_1,w_1)$.

   To do this, we employ the following property of BFS layerings of $t$-trees \cite{KP08,DMW05}:  For every integer $s$,
  for each component $B$ of $H[L_{s+1}]$,
    the set $N_B$ of vertices in $L_{s}$ that are adjacent to at least one vertex in $B$ forms a clique in $H$.
  Since $v_1$ and $w_1$ are in the same component $A$ of $H^\Delta_a$, this implies that $B:=A[L_{r_H+a\Delta}]$ is connected.  This implies that the set $N_B$ of vertices in $L_{r_H+a\Delta-1}$ adjacent to vertices in $B$ form a clique.  Then $H^\Delta_a$ contains a path of length at most $\Delta$ from $v_1$ to a vertex $v_1'$ in $N_B$.   Likewise, $H^\Delta_a$ contains a path of length at most $\Delta$ from $w$ to a vertex $w'$ in $N_B$. Since $N_B$ is a clique, $v'=w'$ or $v'$ and $w'$ are adjacent.  In the former case, there is a path in $H$ from $v$ to $w$ of length at most $2\Delta$. In the latter case there is a path from $v$ to $w$ of length at most $2\Delta+1$.
\end{proof}

\begin{clm}\label{good_probability}
  Fix some vertex $v$ of $H\boxtimes P$ independently of $r_H$ and $r_P$ and let $(a,b)$ be such that $v$ is a vertex of $(H\boxtimes P)^\Delta_{a,b}$.  Then, with probability at least $1/4$,  
  \[
  d_{H\boxtimes P}(v, \overline{(H\boxtimes P)^\Delta_{a,b}} ) \ge \Delta/4 \enspace .
  \]
\end{clm}

\begin{proof}
  Let $v:=(v_1,v_2)$.
  Let $\mathcal{E}$ be the event $d_{H\boxtimes P}(v, V(H\boxtimes P)\setminus V((H\boxtimes P)_{a,b}^\Delta)) \ge \Delta/4$, let $\mathcal{E}_H$ be the event $d_{H}(v_1, \overline{H^\Delta_a} ) \ge \Delta/4$ and let $\mathcal{E}_P$ be the event $d_{P}(v_2, \overline{P^\Delta_b} ) \ge \Delta/4$.  Then $\mathcal{E}=\mathcal{E}_H\cap\mathcal{E}_P$.

  Recall that our partition is defined in terms of a BFS layering $\{L_i:i\in\Z\}$ of $H$ and a random offset $r_H\in\{0,\ldots,\Delta-1\}$.
  Let $i$ be such that $v_1\in L_i$.  The complementary event $\overline{\mathcal{E}_H}$ occurs if and only if $(i\bmod\Delta)-r_H\in\{-\Delta/4-1,\ldots,\Delta/4-1\}$. The number of such $r_H$ is $\Delta/2-1$, so $\Pr(\overline{\mathcal{E}_H})=(\Delta/2-1)/\Delta < 1/2$ and $\Pr(\mathcal{E}_H)>1/2$.  Similarly $\overline{\mathcal{E}_P}$ occurs if and only if $v_2=y_j$ and $|(j\bmod\Delta)-r_P|\in\{-\Delta/4-1,\ldots,\Delta/4-1\}$ which also occurs with probability less than $1/2$, so $\Pr(\mathcal{E}_P)> 1/2$.

  The events $\mathcal{E}_H$ and $\mathcal{E}_P$ are independent since the occurrence of $\mathcal{E}_H$ is determined entirely by the choice of $r_H$ and the occurrence of $\mathcal{E}_P$ is determined entirely by the choice of $r_P$.
  Therefore $\Pr(\mathcal{E})=\Pr(\mathcal{E}_H)\cdot\Pr(\mathcal{E}_P) > 1/4$.
\end{proof}

\paragraph{\boldmath The Coordinate Function $\varphi_I$.}

Let $I$ be the union of the vertex-disjoint graphs $(H\boxtimes P)^\Delta_{a,b}$ over all integers $a$ and $b$.
Thus, $I$ is a random subgraph of $H\boxtimes P$ whose value depends only on the random choices $r_H$ and $r_P$.  For each component $C$ of $I$, let $X_C:=\cup\{X_{i,j}:C\subseteq Q^+_{i,j},\, i\in\{0,\ldots,\log N\}, j\in\{0,\ldots,N/2^i-1\}\}$.  In words, $X_C$ contains only the vertical cuts used to construct $X$ that cut $C$ from top to bottom. Let $J$ be the subgraph of $I$ obtained by removing, for each component $C$ of $I$, the vertices in $X_C\cap V(C)$.

\begin{clm}\label{dstar_component_diameter}
  Each component $C'$ of $J$ has $\diam_{d^*}(C')\le 5\Delta$.
\end{clm}

\begin{proof}
  Let $C'$ be a component of $J$, let $C$ be the component of $I$ that contains $C'$, and let $v$ and $w$ be two vertices of $C'$. Our task is to show that $d^*(v,w)\le 5\Delta$.  By \cref{component_diameter}, $d_{H\boxtimes P}(v,w)\le 2\Delta+1< 5\Delta$, so we may assume that $d^*(v,w)\neq d_{H\boxtimes P}(v,w)$.  Therefore $d^*(v,w)=d_{i,j}(v,w)$ for some $i$ and $j$ such that $v$ and $w$ are in different components of $Q^+_{i,j}-X_{i,j}$.  Since $v$ and $w$ are in the same component $C$ of $I$, the component $C$ is not contained in $Q^+_{i,j}$. (Otherwise, $X_{i,j}$ would be in $X_C$ and $v$ and $w$ would be in different components of $J$.) Therefore $C$ contains a vertex $x$ that is not in $Q^+_{i,j}$.  By \cref{component_diameter}, $d_{H\boxtimes P}(x,v)\le 2\Delta+1$ and $d_{H\boxtimes P}(x,w)\le 2\Delta+1$.  Therefore $d^*(v,w)=d_{i,j}(v,w)\le d_P(v_P,x_P)+d_P(x_P,w_P)\le d_{H\boxtimes P}(v,x)+d_{H\boxtimes P}(x,w)\le 4\Delta+2\le 5\Delta$.
\end{proof}

For each component $C'$ of $J$, choose a uniformly random $\alpha_{C'}$ in $[0,1]$, with all choices made independently.
For each component $C$ of $I$, each component $C'$ of $J$ that is contained in $C$, and each $v \in V(C')$, let
\[
  \varphi_{I}(v):=(1+\alpha_{C'})\,d_{H\boxtimes P}(v,\overline{C}) \enspace .
\]


\begin{obs}\label{uniform}
  Fix $I:=I(H,P,\Delta,r_H,r_P)$ and $J:=J(H,P,G)$.  For each $v\in V(J)$,
  $\varphi_I(v)$ is uniformly distributed in the real interval $[d_{H\boxtimes P}(v,\overline{C}),\;2d_{H\boxtimes P}(v,\overline{C})]$.
\end{obs}



\begin{clm}\label{double_distance}
  For any  $v,w\in V(H\boxtimes P)\setminus X$,
  \[
    |\varphi_I(v)-\varphi_I(w)| \le 2\, d^*(v,w) \enspace .
  \]
\end{clm}

\begin{proof}
  If $v=w$ then $|\varphi_I(v)-\varphi_I(w)|=0 = 2 d^*(v,w)$, so we assume $v\neq w$. In particular $d^*(v,w)\ge d_{H\boxtimes P}(v,w)\ge 1$.  There are three cases to consider, depending on the placement of $v$ and $w$ with respect to the components of $I$ and $J$.

  \begin{compactenum}
    \item 
     If $v$ and $w$ are in different components $C_v$ and $C_w$ of $I$ then, for some $\alpha_v,\alpha_w\in[0,1]$,
    \begin{align*}
       |\varphi_I(v)-\varphi_I(w)|
      & =|(1+\alpha_v)d_{H\boxtimes P}(v,\overline{C_v})-(1+\alpha_w)d_{H\boxtimes P}(w,\overline{C_w})| \\
      & \le 2\max\{d_{H\boxtimes P}(v,\overline{C_v}), d_{H\boxtimes P}(w,\overline{C_w})\}
      - \min\{d_{H\boxtimes P}(v,\overline{C_v}), d_{H\boxtimes P}(w,\overline{C_w})\} \\
      & \le 2\left(d_{H\boxtimes P}(v,\overline{C_v}) + d_{H\boxtimes P}(w,\overline{C_w})\right) - 3\\
      & \le 2d_{H\boxtimes P}(v,w) - 1 \le 2d^*(v,w) \enspace ,
    \end{align*}
    where the penultimate inequality follows from the fact that every path in $H\boxtimes P$ from $v$ to $w$ contains a minimal subpath that begins at $v$ and ends in $\overline{C_v}$ and a minimal subpath begins in $\overline{C_w}$ and ends at $w$.  These two subpaths have at most one edge in common, so
    $d_{H\boxtimes P}(v,\overline{C_v}) + d_{H\boxtimes P}(w,\overline{C_w}) \le d_{H\boxtimes P}(v,w)+1$.  We now  assume that $v$ and $w$ are in the same component, $C$, of $I$.

    \item If $v$ and $w$ are in the same component $C'$ of $J$, then $v$ and $w$ are in the same component $C$ of $I$.  Then
    \begin{align*}
      |\varphi_I(v)-\varphi_I(w)|
      & =(1+\alpha_{C'})|d_{H\boxtimes P}(v,\overline{C})-d_{H\boxtimes P}(w,\overline{C})| \\
      & \le 2|d_{H\boxtimes P}(v,\overline{C})-d_{H\boxtimes P}(w,\overline{C})| \\
      & \le 2 d_{H\boxtimes P}(v,w) \le 2d^*(v,w) \enspace ,
    \end{align*}
    where the penultimate inequality is obtained by rewriting the triangle inequalities $d_{H\boxtimes P}(v,\overline{C})\le d_{H\boxtimes P}(v,w)+d_{H\boxtimes P}(w,\overline{C})$ and $d_{H\boxtimes P}(w,\overline{C})\le d_{H\boxtimes P}(w,v)+d_{H\boxtimes P}(v,\overline{C})$.

    \item It remains to consider the case where $v$ and $w$ are in the same component $C$ of $I$ but in different components $C'_v$ and $C'_w$ of $J$.  This happens because there exists some $i$ and $j$ such that $C$ is contained in $Q^+_{i,j}$ but $v$ and $w$ are in different components of $Q^+_{i,j}-X_{i,j}$. In this case, $d^*(v,w)\ge d_{i,j}(v,w)=d_P(v_P,x)+d_P(x,w_P)$ for some $x\in \overline{P^+_{i,j}}$.  Since $C$ is contained in $Q^+_{i,j}$, $d_{H\boxtimes P}(v,\overline{C})\le d_{P}(v_P,x)$ and $d_{H\boxtimes P}(\overline{C},w)\le d_{P}(x,w_P)$. Therefore $d^*(v,w)\ge d_{H\boxtimes P}(v,\overline{C})+d_{H\boxtimes P}(w,\overline{C})$. Therefore
    \begin{align*}
        |\varphi_I(v)-\varphi_I(w)|
        & =|(1+\alpha_{C'_v})d_{H\boxtimes P}(v,\overline{C})-(1+\alpha_{C'_w})d_{H\boxtimes P}(w,\overline{C})| \\
        & \le 2\max\{d_{H\boxtimes P}(v,\overline{C}), d_{H\boxtimes P}(\overline{C},w)\} \\
        & \le 2\left(d_{H\boxtimes P}(v,\overline{C}) + d_{H\boxtimes P}(\overline{C},w)\right) \\
        & \le 2\, d^*(v,w) \enspace . \qedhere
    \end{align*}
  \end{compactenum}
\end{proof}

\paragraph{\boldmath The Euclidean Embedding $\phi$.}
\label{Euclidean_embedding_section}


Let $a>0$ be a constant whose value will be bounded from below later.  We now define a random function $\phi:V((H\boxtimes P)-X)\to\R^{L}$ where $L:=\lfloor 1+\log n\rfloor\cdot\lceil a k\ln n\rceil$. For each $i\in\{0,\ldots,\log N-1\}$ and each $j\in\{1,\ldots,\lceil a k\ln n\rceil\}$, let $I_{i,j}:=I(H,P,2^i,r_{H,i,j},r_{P,i,j})$ be an instance of the random subgraph $I$ defined in the previous section with parameter $\Delta=2^i$ and where random offsets $r_{H,i,j},r_{P,i,j}\in\{0,\ldots,\Delta-1\}$ are chosen independently for each instance.  From each $I_{i,j}$ and the sets $\{X_{i',j'}:i'\in\{0,\ldots,\log N-1\},\, j'\in\{1,\ldots,N/2^i-1\}\}$, we  define the subgraph $J_{i,j}$ of $I_{i,j}$ as in the previous section.  This defines a uniformly random $\alpha_{C'}$ for each component $C'$ of $J_{i,j}$, with all random choices made independently.  This defines, for each $v\in V(J_{i,j})$, the value $\varphi_{I_{i,j}}(v)$ and we let $\mathdefin{\phi_{i,j}(v)}:= \varphi_{I_{i,j}}(v)$.

Finally, define the Euclidean embedding $\phi:V((H\boxtimes P)-X)\to\R^L$ as
\[
   \phi(x) := \left(\phi_{i,j}(x):(i,j)\in \{0,\ldots,\lfloor \log n\rfloor\}\times\{1,\ldots,\lceil a k\ln n\rceil\}\right) \enspace .
\]

The following lemma says that $\phi/2\sqrt{L}$ is a Euclidean contraction of $(V((H\boxtimes P)- X),d^*)$.  In a final step, we divide each coordinate of $\phi$ by $2\sqrt{L}$ to obtain an Euclidean contraction. Until then, it is more convenient to work directly with $\phi$.


\begin{clm}\label{euclidean_contraction}
  For each $v,w\in V((H\boxtimes P)-X)$, 
  \[
  d_2(\phi(v),\phi(w)) \le 2\sqrt{L}\cdot d^*(v,w) \enspace .\]
\end{clm}

\begin{proof}
  By \cref{double_distance}, $|\phi_{i,j}(v)-\phi_{i,j}(w)|\le 2d^*(v,w)$ for each $(i,j)\in\{0,\ldots,\log N\}\times\{1,\ldots,\lceil ak\ln n\rceil\}$.  Therefore,
  \[
    d_2(\phi(v),\phi(w)) = \left(\sum_{i,j}(\phi_{i,j}(v)-\phi_{i,j}(w))^2\right)^{1/2}
    \le \left(L (2d^*(v,w))^2\right)^{1/2} = 2\sqrt{L}\cdot d^*(v,w) \enspace . \qedhere
  \]
\end{proof}

The remaining analysis in this section closely follows \citet{rao:small}, which in turn closely follows \citet{feige:approximating}.  The main difference is that we work with $d^*$ rather than $d_G$.  We proceed slowly and carefully since our setting is significantly different, and we expect that many readers will not be familiar with some methods introduced by \citet{feige:approximating} that are only sketched by \citet{rao:small}. We make use of the following simple \defin{Chernoff Bound}:  For a $\operatorname{binomial}(n,p)$ random variable $B$, $\Pr(B \le np/2) \le \exp(-np/8)$.

Let $\Gamma_k:=\{(\lambda_1,\ldots,\lambda_k)\in \R^k:\sum_{j=1}^k\lambda_j=1\}$; that is, $\Gamma_k$ is the set of coefficients that can be used to obtain an affine combination of $k$ points.  The following lemma is the crux of the proofs in \cite{rao:small,feige:approximating}. We emphasize that in this lemma, probability is with respect to the random choices made when constructing the embedding $\phi$.  The vertices $v_1,\ldots,v_p$ are arbitrary (not necessarily random) vertices of $(H\boxtimes P)-X$.  In this lemma, it is critical that the function $\lambda$ chooses an affine combination $\lambda_1,\ldots,\lambda_{p-1}$ by only considering $\phi(v_1),\ldots,\phi(v_{p-1})$.  Thus any dependence between $\lambda_1,\ldots,\lambda_{p-1}$ and $\phi(v_p)$ is limited to the random choices made during the construction of $\phi$ that contribute to $\phi(v_1),\ldots,\phi(v_{p-1})$.

\begin{clm}\label{crux} 
  Fix some function $\lambda:(\R^{L})^{p-1}\to \Gamma_{p-1}$.
  Let $v_1,\ldots,v_p$ be distinct vertices of $(H\boxtimes P)-X$ and let $h:=d^*(v_p,\{v_1,\ldots,v_{p-1}\})$.  Let $(\lambda_1,\ldots,\lambda_{p-1}):=\lambda(\phi(v_1),\ldots,\phi(v_{p-1}))$ and let $x:=\sum_{j=1}^{p-1}\lambda_j\phi(v_j)$.
  Then, for all $a\ge 193$, $n\ge 2$, and $k\ge 2$,
  \[
    d_2(\phi(v_p),x)\ge \frac{h\sqrt{\lceil ak\ln n\rceil}}{640\sqrt{2}} \enspace ,
  \]
  with probability at least $1-n^{-3k}$.
\end{clm}

\begin{proof}
  If $h \le 5$ then let $i:=0$.  Otherwise, let $i$ be the unique integer such that $h/10\le 2^i < h/5$. Let $\Delta:=2^i$. To prove the lower bound on $d_2(\phi(v_p),x)$, we will only use the coordinates $\phi_{i,1},\ldots,\phi_{i,\lceil a k\ln n\rceil}$.  For each $j\in\{1,\ldots,\lceil ak\ln n\rceil\}$, let $C_{i,j}$ and $C'_{i,j}$ be the components of $I_{i,j}$ and $J_{i,j}$, respectively, that contain $v_p$. We say that $j\in\{1,\ldots,\lceil a k\ln n\rceil\}$ is \defin{good} if $d_{H\boxtimes P}(v_p,\overline{C_{i,j}})\ge \Delta/4$.  By \cref{good_probability},  $\Pr(\text{$j$ is good})\ge 1/4$. Let $S:=\{j\in\{1,\ldots,\lceil a k\ln n\rceil\}:\text{$j$ is good}\}$.  Since $I_{i,1},\ldots,I_{i,\lceil a k\ln n\rceil}$ are mutually independent, $|S|$ dominates\footnote{We say that a random variable $X$ \defin{dominates} a random variable $Y$ if $\Pr(X\ge x)\ge\Pr(Y\ge x)$ for all $x\in\R$.} a \linebreak $\operatorname{binomial}(\lceil a k\ln n\rceil,1/4)$ random variable. By the Chernoff Bound,
  \[
  \Pr(|S|\ge \tfrac{1}{8}\lceil a k\ln n\rceil)\ge 1-\exp(-(ak\ln n)/32) \enspace .
  \]

  By \cref{uniform}, $\phi_{i,j}(v_p)$ is uniformly distributed over an interval of length at least $\Delta/4$, for each $j\in S$.  We claim that the location of $\phi_{i,j}(v_p)$ in this interval is independent of the corresponding coordinate, $x_{i,j}$, of $x$.  If $\Delta=1$, then $v_p$ is the only vertex in $C'_{i,j}$.  Otherwise, since $\Delta<h/5$, \cref{dstar_component_diameter} implies that $C'_{i,j}$ does not contain any of $v_1,\ldots,v_{p-1}$. In either case, $C'_{i,j}$ does not contain any of $v_1,\ldots,v_{p-1}$.  Therefore, the location of $\phi_{i,j}(v_p)$ is determined by a random real number $\alpha_{i,j}:=\alpha_{C'_{i,j}}\in[0,1]$ that does not contribute to $\phi(v_1),\ldots,\phi(v_{p-1})$.  Since $(\lambda_1,\ldots,\lambda_{p-1})=\lambda(\phi(v_1),\ldots,\phi(v_{p-1}))$ is completely determined by $\phi(v_1),\ldots,\phi(v_{p-1})$, it follows that $\alpha_{i,j}$ is independent of $x=\sum_{k=1}^{p-1}\lambda_k\phi(v_k)$. In particular, $\alpha_{i,j}$ is independent of $x_{i,j}$.

  Therefore, for $j\in S$, $\Pr(|\phi_{i,j}(v_p)-x_{i,j}|\ge \Delta/16)\ge 1/2$.\footnote{The coordinate $\phi_{i,j}(v_p)$ is uniform over some interval $[a,b]$ of length $b-a\ge \Delta/4$ whereas $[x_{i,j}-\Delta/16,x_{i,j}+\Delta/16]$ has length $\Delta/8$, so $\Pr(|\phi_{i,j}(v_p)-x_{i,j}|\ge \Delta/16)\ge (b-a-\Delta/8)/(b-a)\ge 1/2$.}
  Let $S':=\{j\in S:  |\phi_{i,j}(v_p)-x_{i,j}|\ge \Delta/16\}$.  Then $|S'|$ dominates a $\operatorname{binomial}(|J|,1/2)$ random variable.  By the Chernoff Bound (and the union bound),
  \[
    \Pr(|S'|\ge \tfrac{1}{32}\lceil a k\ln n\rceil)\ge 1-\exp(-ak\ln n/64)-\exp(-ak\ln n/32)\ge 1-n^{-3k} \enspace ,
  \]
  for all $a\ge 193$, $n\ge 2$, and $k\ge 2$.
  Therefore,
  \begin{align*}
    d_2(\phi(v_p),x)
    & = \left(\sum_{i'=0}^{\lfloor\log n\rfloor}\sum_{j=1}^{\lceil ak\ln  n\rceil}(\phi_{i',j}(v_p)-x_{i',j})^2\right)^{1/2} \\
    & \ge \left(\sum_{j=1}^{\lceil ak\ln  n\rceil}(\phi_{i,j}(v_p)-x_{i,j})^2\right)^{1/2} \\
    & \ge \left(\sum_{j\in S'}(\Delta/16)^2\right)^{1/2} \\
    & \ge \left((\Delta/16)^2\cdot \tfrac{1}{32}\lceil ak\ln  n\rceil\right)^{1/2}
      & \text{(with probability at least $1-n^{-3k}$)} \\
    & = \frac{\Delta\sqrt{\lceil ak\ln n\rceil}}{64\sqrt{2}} \\
    & \ge \frac{h\sqrt{\lceil ak\ln n\rceil}}{640\sqrt{2}}
     & \text{(since $\Delta\ge h/10$)}. &
    \qedhere
  \end{align*}
\end{proof}

A variant of the following lemma is proven implicitly by \citet[pages~529--530]{feige:approximating}.  For completeness, we include a proof in \cref{volume_preserver_proof}.

\begin{restatable}{clm}{volumepreserverClaim}
\label{volume_preserver}
  For every $k$-element subset $K$ of $V((H\boxtimes P)-X)$,
  \[
    \Pr\left(\evol(\phi(K)) \ge \frac{\tvol_{d^*}(K)\cdot(2\zeta/3)^{k-1}}{(k-1)!}\right) \ge 1- O(kn^{-k}) \enspace .
  \]
  where $\zeta:=\sqrt{\lceil ak\ln n\rceil}/(640\sqrt{2})$ is the expression that also appears in \cref{crux}.
\end{restatable}

We now have all the pieces needed to complete the proof of \cref{dstar_contraction}.

\begin{proof}[Proof of \cref{dstar_contraction}]
  For each $v\in V(H\boxtimes P)$, let $\phi'(v):=\phi(v)/2\sqrt{L}$. By \cref{dstar_contraction}, $\phi'$ is a Euclidean contraction of $\mathcal{M}^*$.  By \cref{volume_preserver}, for each $K\in \binom{V(G)}{k}$,
  \begin{equation}
    \Pr\left(\evol(\phi'(K)) \ge \frac{\tvol_{d^*}(K)\cdot(2\zeta/3)^{k-1}}{(k-1)!(2\sqrt{L})^{k-1}}\right) \ge 1- O(kn^{-k}) \enspace .
    \label{zippy}
  \end{equation}
  By the union bound, the probability that the volume bound in \cref{zippy} holds for every $K\in\binom{V(G)}{k}$ is at least $1-O(\binom{n}{k}kn^{-k}) > 0$ for all sufficiently large $n$.  When this occurs,
  \[
    \evol(\phi'(K)) \ge \frac{\tvol_{d^*}(K)\cdot(2\zeta/3)^{k-1}}{(k-1)!(2\sqrt{L})^{k-1}} \ge
    \frac{\ivol_{d^*}(K)\cdot(2\zeta/3)^{k-1}}{(2\sqrt{L})^{k-1}}
    =
    \frac{\ivol_{d^*}(K)\cdot\zeta^{k-1}}{(3\sqrt{L})^{k-1}}
    \enspace ,
  \]
  by \cref{tvol_vs_ivol}.  Then $\phi'$ is a $(k,\eta)$-volume-preserving contraction for
  \[
    \eta = \frac{3\sqrt{L}}{\zeta} = \frac{3\cdot 640\sqrt{2L}}{\sqrt{ak\ln n}} = {1920\sqrt{2\floor{1+\log n}}} \in 
    O(\sqrt{\log n}) \enspace . \qedhere
  \]
\end{proof}

We now complete the proof of the main result of this section. 


\begin{proof}[Proof of \cref{rtw-flex}] 
Let $G$ be a graph with row treewidth $b$. 
We may assume that $G$ is connected.
By assumption, $G$ is contained in $H\boxtimes P$ for some graph $H$ with treewidth $b$ and for some path $P$.
We may assume without loss of generality that $H$ is a $b$-tree. 
For simplicity, we assume $G$ is a subgraph of $H\boxtimes P$. 
Let $\delta\in\R_{\geq 1}$ be any multiplier. 
Let $X\subseteq V(H\boxtimes P)$ be the set defined in \cref{x_definition}, so $|X|\in O((bn\log n)/\delta)$. 
Let $d^*$ (which depends on $X$) be the distance function defined in \cref{d_star_definition}.  By \cref{supergraph_contraction,d_star_summary}, the metric space $\mathcal{M}^*:=(V(G)\setminus X,d^*)$ is a contraction of the graphical metric $\mathcal{M}_{G-X}$, and
$\mathcal{M}^*$ has local density at most $\delta$. Let $k:=\ceil{\log n}$ and $\eta:=\sqrt{\log n}$.
By \cref{dstar_contraction}, $\mathcal{M}^*$ has a $(k,O(\eta))$-volume-preserving Euclidean contraction.  Therefore, by \cref{volume_density_bandwidth}, 
$\bw(\mathcal{M}^*)\in
O( (nk\log\Delta)^{1/k} \delta k\eta \log^{3/2} n)$.
Since $\Delta\leq n$ and $k=\ceil{\log n}$, we have
$(nk\log\Delta)^{1/k}\in O(1)$.
Thus
$\bw(\mathcal{M}^*)\in O(  \delta k\eta \log^{3/2} n) \in O( \delta \log^3 n)$.
By \cref{contraction_increases_bandwidth}, 
$\bw(G-X) =\bw(\mathcal{M}_{G-X})\le \bw(\mathcal{M}^*) \in
O(\delta \log^3 n)$.
\end{proof}

\section{\boldmath $K_h$-Minor-Free Graphs}
\label{KtMinorFree}

This section proves \cref{minor-fanblowup} for $K_h$-minor-free graphs. The starting point is the tree-decomposition in \cref{GMST}.  
The proof then consists of two steps. The first `splits the tree-decomposition', and the second `processes the smaller sections'. Step~1 is primarily handled by \cref{mainLem}, Step~2 by \cref{starDecomp}, and the other results assist in one of these two steps. We now explain the ideas behind these steps.

For a graph class $\GG$, let \defin{$\widehat{\GG}$} be the class of graphs $G$ such that there exists a set $Z\subseteq V(G)$ of degree 1 vertices in $G$ such that $G-Z\in \GG$. For a class of graphs $\GG$ and integer $a\geq 0$, let \defin{$\Apex{\GG}{a}$} be the class of graphs $G$ such that $G-X\in \GG$ for some $X\subseteq V(G)$ with $|X|\leq a$. We call $X$ the \defin{apices} of $G$. Observe that if $\GG$ is closed under adding isolated vertices, then $\widehat{\Apex{\GG}{a}}\subseteq \Apex{\widehat{\GG}}{a}$.

    Starting with a $K_h$-minor-free graph $G$, let $b,k$ be from \cref{GMST}, let $\scr{H}$ be the class of all graphs with row treewidth at most $b$, and let $a:=\max(h-5,0)$. By \cref{GMST}, there is a tree-decomposition $(B_t:t\in V(T))$ of $G$ of adhesion at most $k$ such that each torso is in $\Apex{\scr{H}}{a}$. Our goal is to find a small set of vertices whose removal breaks the graph, and the tree-decomposition, into smaller chunks. To do this, we use the same method as \citet{distel.dujmovic.ea:product}. In particular, 
    \cref{treeDeletions,mainLem} give a small set $Z\subseteq V(T)$ whose deletion splits $T$ into subtrees $T'$ such that the total number of vertices contained in $\bigcup_{t\in V(T')}B_t$ is small. 
    
    Consider $T$ to be rooted at a vertex $r$. For each $z\in Z\setminus \{r\}$, delete the at most $k$ vertices in $B_z$ that are in the bag of $z$'s parent. The total number of vertices deleted is small, but it has the effect of breaking the graph and tree-decomposition into manageable chunks. Specifically, each chunk has a star-decomposition $(B'_s:s\in V(S))$ of  adhesion at most $k$, such that if $r'$ is the centre of $S$, then the torso at $r'$ is in $\Apex{\scr{H}}{a}$, and each bag $B'_s$, $s\neq r'$, is small. 
We then process these chunks separately, showing bandwidth-flexibility for each chunk. Using this flexibility, we can guarantee that the total number of `removed' vertices across all chunks is still small, while maintaining that the bandwidth of the non-removed vertices in each chunk is small with respect to $|V(G)|$. 

We now sketch Step 2, where the smaller chunks are processed. Since the central torso is in $\Apex{\scr{H}}{a}$, the majority of the work is done by \cref{rtw-flex}, since we can just add the apex vertices to the set of removed vertices; see \cref{addApex}. However, we still need to attach the vertices in the smaller bags $B'_s$, $s\neq r$. For this, we create an auxiliary graph $G'$, obtained from the central torso by attaching degree-1 vertices. Specifically, for each vertex $v$ in some bag $B'_s$, $s\neq r$, and each vertex $x$ in $B'_s\cap B'_r$, we add an auxiliary vertex $v_x$ adjacent to $x$. See \cref{FigDeg1}. The resulting graph is in $\widehat{\Apex{\scr{H}}{a}}$. However, $\widehat{\Apex{\scr{H}}{a}}=\Apex{\scr{H}}{a}$ as $\widehat{\scr{H}}=\scr{H}$ (see \cref{rtwExt}) and since $\scr{H}$ is closed under adding isolated vertices. \cref{rtw-flex,addApex} then establish bandwidth-flexibility for $G'$, where $X$ is the set of removed vertices.
  
For each $B'_s$, $s\neq r$, one of two things must happen. One possibility is that $B'_s\cap B'_r\subseteq X$, in which case $B'_s$ is easily accommodated in the small bandwidth ordering. The second possibility is that there exists some $x\in (B'_s\cap B'_r)\setminus X$. In this case, for each vertex $v\in B'_s\setminus B'_r$, one of two outcomes must have occurred. Either $v_x\in X$, in which case $v$ is added to the set of removed vertices, or $v_x$ is at distance at most $2$ in $G'-X$ from every vertex in $(B'_s\cap B'_r)\setminus X$ and every vertex $v'_x$ with $v'\in B'_s\setminus B'_r$ and $v'_x\notin X$. In particular, these vertices are adjacent in $(G'-X)^2$. Observe that\footnote{For $d\in\NN_1$, \defin{$G^d$} is the \defin{$d$-th power} of $G$, which is the graph with vertex set $V(G)$, with $vw\in E(G^d)$ if and only if $1\leq\dist_G(v,w)\leq d$. If $\bw_G(v_1,\dots,v_n)\leq k$ then $\bw_{G^d}(v_1,\dots,v_n)\leq kd$. So $\bw(G^d)\leq d\,\bw(G)$.} the bandwidth of $(G'-X)^2$ is within a constant factor of the bandwidth of $G'-X$. So we can use a small bandwidth ordering of $(G'-X)^2$ to determine a small bandwidth ordering of $G-X$, as desired. 

We now start the formal proof of \cref{minor-fanblowup}.

    \begin{lem}
        \label{rtwExt}
        For every class of graphs $\GG$, $\rtw(\widehat{\GG})\leq \max(\rtw(\GG),1)$.
    \end{lem}

    \begin{proof}
        Let $b:=\max(\rtw(\GG),1)$. Fix $G\in \widehat{\GG}$. So there exists a set $Z$ of degree 1 vertices in $G$ such that $G-Z\in \GG$. Thus, there exists a graph $H$ of treewidth at most $b$ and a path $P$ such that $G-Z$ is contained in $H\boxtimes P$. Let $E_z$ be the set of edges of $G$ with an endpoint in $Z$. We may consider each edge in $E_z$ to be of the form $(h,p)z$ with $h\in V(H)$, $p\in V(P)$, $z\in Z$.
        
        Let $H'$ be the graph with vertex set $V(H)\cup Z$ and edge set $E(H)\cup \{hz:(h,p)z\in E_z\}$. Observe that $H'$ is obtained from $H$ by adding some number of degree 1 vertices. Thus, $\tw(H')\leq \max(\tw(H),1)\leq b$. Next, notice that for each $(h,p)z\in E_z$, we have $hz\in E(H')$ and thus $(h,p)(z,p)\in E(H'\boxtimes P)$. Therefore, $G$ is isomorphic to a subgraph of $H'\boxtimes P$, via the mapping $(h,p)\mapsto (h,p)$ for $v=(h,p)\in V(G-Z)$, and $z\mapsto (z,p)$ for the unique $(h,p)\in V(G-Z)$ such that $(h,p)z\in E_z$. This gives the desired result.
    \end{proof}

\begin{lem}
\label{addApex}
Let $f,g:\NN\mapsto \R^+$ be non-decreasing functions, let $a\in \NN$, and let $\GG$ be a class with  $(f,g)$-bandwidth. Then $\Apex{\GG}{a}$ has $(f+a,g)$-bandwidth.
\end{lem}

\begin{proof}
Fix $G\in \Apex{\GG}{a}$, and let $n:=|V(G)|$. There is a set $X_1$ of at most $a$ vertices in $G$ such that $G-X_1\in \GG$. Since $\GG$ has $(f,g)$-bandwidth, there is a set $X_2\subseteq V(G-X_1)$ such that $|X_2|\leq f( |V(G-X_1-X_2)| )\leq f(n)$ and $\bw(G-X_1-X_2)\leq g(|V(G-X_1-X_2)|) \leq g(n)$. Here, we use that $f$ and $g$ are non-decreasing. Let $X':=X_1\cup X_2$. Thus $|X'|\leq f(n)+a$ and $\bw(G-X')\leq g(n)$. Hence $\Apex{\GG}{a}$ has $(f+a,g)$-bandwidth.
 \end{proof}

We also use the following standard separator lemma, similar in spirt to \cref{weighted_separator}; see \citet[Lemma~9]{distel.dujmovic.ea:product} for a proof. 


\begin{lem}
\label{treeDeletions}
For each $q\in\NN$ and $n\in \R^+$, every vertex-weighted  tree $T$ with total weight at most $n$ has a set $Z$ of at most $q$ vertices such that each component of $T-Z$ has weight at most $\frac{n}{q+1}$. 
\end{lem}




The next lemma corresponds to Step~2 in the above sketch. 
The proof uses the following observation: for any graph $G$ and $w\in \NN$, $\bw(G)\leq w$ if and only if there exists an injective function $\sigma:V(G)\mapsto \NN$ such that $|\sigma(u)-\sigma(v)|\leq w$ for each edge $uv\in E(G)$. 



    \begin{lem}
        \label{starDecomp}
        Let $k,n\in \NN$, let $w\in \R^+$, let $f:\NN\mapsto \R^+$ and $g:\NN\mapsto \R_{\geq 1}$ be non-decreasing functions, and let $\scr{H}$ be a class of graphs such that $\widehat{\scr{H}}$ has $(f,g)$-bandwidth. Let $G$ be an $n$-vertex graph that admits a star-decomposition $(B_s:s\in V(S))$ of adhesion at most $k$ such that, if $r$ is the centre of $S$, then:
        \begin{itemize}
            \item $\UWT{G}{B_r}\in \scr{H}$, and
            \item for each $s\in V(S)\setminus \{r\}$, $|B_s\setminus B_r|\leq w$.
        \end{itemize}
        Then $G$ has $(f(kn),\max(2g(kn),w))$-bandwidth.
    \end{lem}

    \begin{proof}
        For each $s\in V(S)\setminus \{r\}$, let $K_s:=B_s\cap B_r$ and $B_s':=B_s\setminus B_r$. Note that $|K_s|\leq k$ and $|B_s'|\leq w$. Let $G'$ be the graph obtained from $\UWT{G}{B_r}$ by adding, for each $s\in V(S)\setminus \{r\}$, each $v\in B_s'$, and each $x\in K_s$, a new vertex $v_x$ adjacent only to $x$; see \cref{FigDeg1}. Observe that $G'\in \widehat{\scr{H}}$, and that $|V(G')|\leq |B_r|+k|V(G)\setminus B_r|\leq k|V(G)|=kn$. For each $s\in V(S)\setminus \{r\}$ and each $v\in B_s'$, let $M_v:=\{v_x:x\in K_s\}$.
        \begin{figure}[!ht]
        \begin{center}
        \begin{minipage}{0.85\textwidth}
        \centering
        \includegraphics[width=\textwidth]{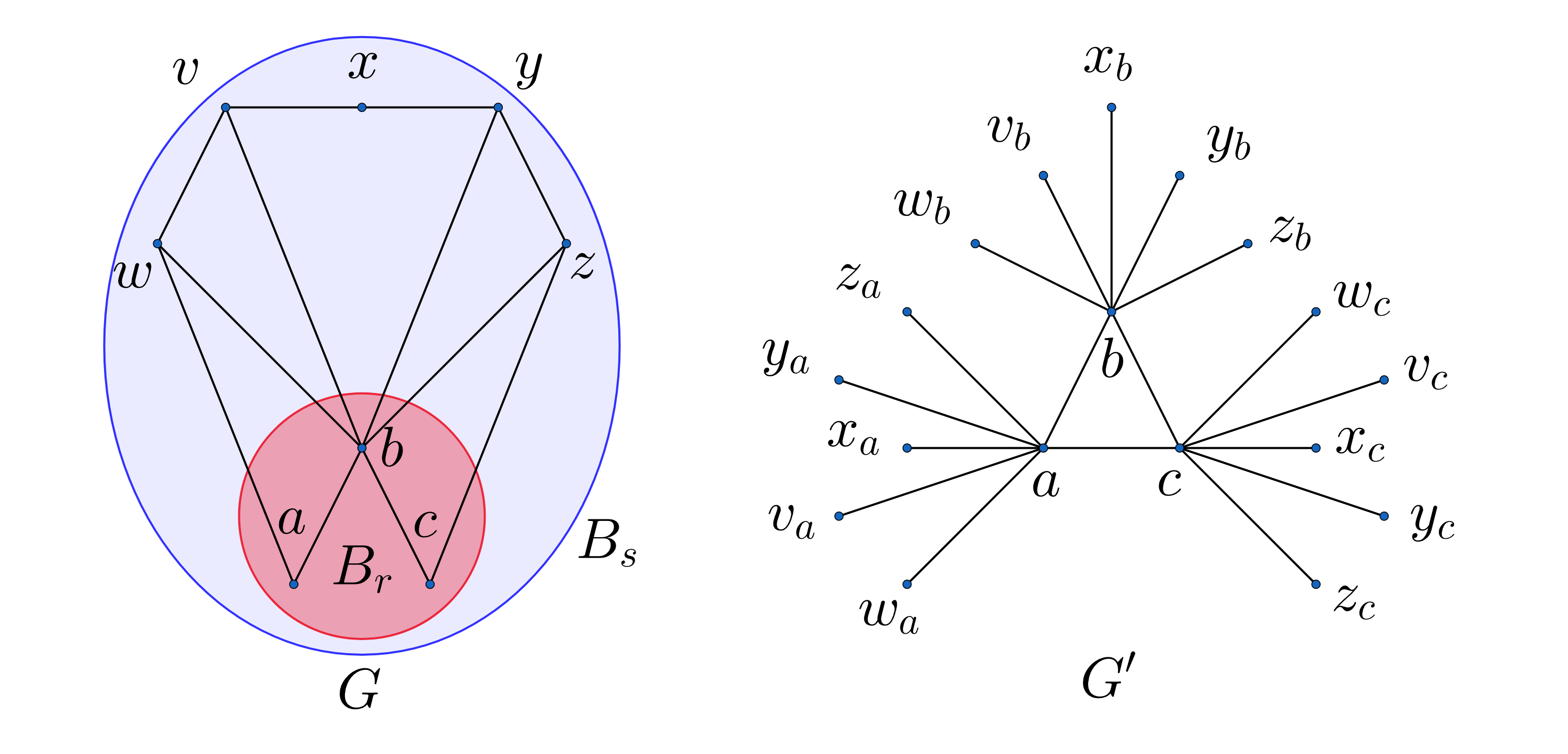}
        \caption{\label{FigDeg1}A diagram of how $G$ is turned into $G'$. The left graph is $G$, the right graph is $G'$. On the left, the vertices in the red section are $B_r$, where $r$ is the centre. The vertices in the blue section (all vertices) are $B_s$. Note that, on the right, each vertex in $B_s\setminus B_r=B_s'$ has been replaced by three new vertices, each adjacent to a different vertex of $B_r=\{a,b,c\}$, and that $B_r$ has become a clique. In this case, $M_v=\{v_a,v_b,v_c\}$ (and similar for $M_w,M_x,M_y,M_z$).}
        \end{minipage}
        \end{center}
        \end{figure}
        
        By assumption, $G'$ has $(f(|V(G')|),g(|V(G'))$-bandwidth. So there exists $X\subseteq V(G')$ such that $|X|\leq f(|V(G')|)\leq f(kn)$ and $\bw(G'-X) \leq g(|V(G')|)\leq g(kn)$. Here, we use that $f,g$ are non-decreasing. Let $X'_1:=X\cap B_r$, let $X'_2:=\{v\in V(G)\setminus B_r:M_v\cap X\neq \emptyset\}$, and let $X':=X'_1\cup X'_2$. Observe that $|X'_2|\leq |X\setminus B_r|$, thus $|X'|\leq |X|\leq f(kn)$. Also, note that $X'\cap B_r=X\cap B_r=X'_1$.

        We seek to show that $\bw(G-X')\leq \max(2g(kn),w)$. The result follows since $|X'|\leq f(kn)$.
        
        Let $U:=\{s\in V(S)\setminus \{r\}:K_s\setminus X=\emptyset\}$ and $W:=\{s\in V(S)\setminus \{r\}:K_s\setminus X\neq \emptyset\}$. Let $V_U:=(\bigcup_{s\in U}B_s')\setminus X'$ and $V_W:=(\bigcup_{s\in W}B_s')\setminus X'$. For each $s\in U$, define $P_s:=B_s'\setminus X'$. Note that $|P_s|\leq |B_s'|\leq w$. For each $s\in W$, pick some $x\in K_s\setminus X$, set $x_s:=x$ and, for each $v\in B_s'\setminus X'$, set $a_v:=v_x$. Note that as $v\notin X'\supseteq X'_2$, $M_v\cap X=\emptyset$. Hence, $a_v\in M_v$ is not in $X$. For $v\in B_r\setminus X'$, set $a_v:=v$. Note that $a_v\notin X$ as $X\cap B_r=X'\cap B_r$.

        Consider a connected component $C$ of $G-X'$. By definition of $U$ and since $X'\cap B_r=X\cap B_r$, observe that $V(C)$ is either contained in $(B_r\setminus X')\cup V_W$, or in $P_s'$ for some $s\in U$. Thus, to show that $\bw(G'-X)\leq \max(2g(kn),w)$, it suffices to show that $\bw((G-X')[(B_r\setminus X')\cup V_W])\leq \max(2g(kn),w)$, and that, for each $s\in U$, $\bw((G-X')[P_s'])\leq \max(2g(kn),w)$. However, the latter is trivially true since $|P_s'|\leq w$  for each $s\in U$. So we focus on the former statement.

Since $\bw(G'-X)\leq g(kn)$, there exists an injective function $\sigma:V(G'-X)\mapsto \NN$ such that $|\sigma(u)-\sigma(v)|\leq g(kn)$ whenever $uv\in E(G'-X)$. By the triangle inequality, $|\sigma(u)-\sigma(v)|\leq 2g(kn)$ whenever $uv\in E((G'-X)^2)$. Define $\sigma':(B_r\setminus X')\cup V_W\mapsto \NN$ via $\sigma'(v)=\sigma(a_v)$. Note that this is well-defined, as for $v\in (B_r\setminus X')\cup V_W$, $a_v\in V(G')$ exists, and $a_v\notin X$. Further, $\sigma'$ is injective, as $\sigma$ is injective as $a_u\neq a_v$ whenever $u\neq v$ (as $M_u$ and $M_v$ are disjoint whenever $u,v\in V_w$ are distinct, and disjoint to $B_r$). We claim that $|\sigma'(u)-\sigma'(v)|\leq 2g(kn)$ whenever $uv\in E((G-X')[(B_r\setminus X')\cup V_W])$. This implies that $\bw((G-X')[(B_r\setminus X')\cup V_W])\leq \max(2g(kn),w)$.

        We consider three cases.

        Case 1. $u,v\in B_r\setminus X'$:

Then $a_u=u$ and $a_v=v$. Since $u,v\in B_r\setminus X'=B_r\setminus X$ and since $uv\in E(G)$, $uv\in E(G'-X)$. Thus, $|\sigma'(u)-\sigma'(v)|=|\sigma(u)-\sigma(v)|\leq g(kn)$, as desired.

        Case 2. $u,v\in V_W$:

        Note that $u\notin B_r$ and $v\notin B_r$. Since $uv\in E((G-X')[(B_r\setminus X')\cup V_W])\subseteq E(G)$ but $u\notin B_r$ and $v\notin B_r$, there exists exactly one $s\in V(S)\setminus \{r\}$ such that $u,v\in B_s'$. Since $u,v\in V_W$, $s\in W$. Now, if $x:=x_s$, recall that $a_v=v_x$, which is adjacent to $x$ in $G'$, and that $a_u=u_x$, which is adjacent to $x$ in $G'$. Further, recall that $x\notin X$, $a_u\notin X$, and that $a_v\notin X$. Thus, $a_u$ and $a_v$ are at distance at most $2$ in $G'-X$, and adjacent in $(G'-X)^2$. Hence, $|\sigma'(u)-\sigma'(v)|=|\sigma(a_u)-\sigma(a_v)|\leq 2g(kn)$, as desired.

        Case 3. Exactly one of $u,v\in V_W$, and the other is in $B_r\setminus X'$:

        Without loss of generality, $u\in B_r\setminus X'$ and $v\in V_W$. Note that $v\notin B_r$. As in Case 1, $a_u=u\in B_r\setminus X'=B_r\setminus  X$. Similarly to Case 2, since $uv\in E((G-X')[(B_r\setminus X')\cup V_W])\subseteq E(G)$ but $v\notin B_r$, there exists exactly one $s\in V(S)$ such that $u,v\in B_s$, and this $s$ satisfies $s\neq r$, $s\in W$, $v\in B_s'$, and $v\notin B_q'$ for $q\in V(S)\setminus \{r,s\}$. Note also that $u\in K_s$.
        
        If $x:=x_s$, we then have that $a_v=v_x$, which is adjacent to $x$ in $G'$. Recall that $x\notin X$, that $x\in K_s$, and that $a_v\notin X$. Note also that since $u,x\in K_s$, $u$ and $x$ are adjacent in $G'$ by definition of $G'$. Thus, since $\{u,x,a_v\}\cap X=\emptyset$, since $u$ and $x$ are adjacent in $G'$, and since $x$ and $a_v$ are adjacent in $G'$, $a_u=u$ and $a_v$ are at distance at most $2$ in $G'-X$, and are adjacent in $(G'-X)^2$. Hence, $|\sigma'(u)-\sigma(v)|=|\sigma(u)-\sigma(a_v)|\leq 2g(kn)$, as desired.

        Thus, $\bw(G-X')\leq \max(2g(kn),w)$, and $G$ has $((f(kn),\max(g(kn),w))$-bandwidth, as desired.
    \end{proof}
    
\begin{lem}
\label{mainLem}
For any $h\in\NN_1$ there exist $k,b\in\NN_1$ such that the following holds. 
Let $f:\NN\mapsto \R_{\geq 0}$ and
$g:\NN\mapsto \R_{\geq 1}$ 
be functions with $f$ superadditive and $g$ non-decreasing, such that the class of graphs of row treewidth at most $b$ is $(f,g)$-bandwidth-flexible. Then the class of $K_h$-minor-free graphs is $(f',g')$-bandwidth-flexible, where $f'(n):=f(kn)+(2\max(h-5,0)+k)n$ and $g'(n):=2g(kn)$.
\end{lem}

\begin{proof}
Let $k,b\in\NN_1$ be from \cref{GMST} (depending on $h$ only). Since $f$ is superadditive and has non-negative outputs, it is also non-decreasing.
    
Let $\GG$ denote the class of graphs with row treewidth at most $b$, and let $a:=\max(h-5,0)$. By \cref{rtwExt}, $\widehat{\GG}=\GG$. Since $\GG$ is closed under adding isolated vertices, $\widehat{\Apex{\GG}{a}}\subseteq \Apex{\widehat{\GG}}{a}=\Apex{\GG}{a}$. Thus, $\widehat{\Apex{\GG}{a}}=\Apex{\GG}{a}$. 
    
Let $G$ be a $K_h$-minor-free graph, set $n:=|V(G))|$, and let $\delta\in\R_{\geq1}$ be any multiplier. By assumption, $\GG$ has $(f/\delta,g\delta)$-bandwidth. Thus, by \cref{addApex}, $\Apex{\GG}{a}$ has $(f/\delta+a,g\delta)$-bandwidth. Since $\GG$ is monotone, so is $\Apex{\GG}{a}$.


        
Let $(B_t:t\in V(T))$ be a tree-decomposition of $G$ produced by \cref{GMST}. Thus, $(B_t:t\in V(T))$ has adhesion at most $k$, and each torso is in $\Apex{\GG}{a}$. Fix a root $r\in V(T)$. For each $t\in V(T)\setminus \{r\}$, let $p(t)$ denote the parent of $t$, and let $K_t:=B_t\cap B_{p(t)}$. Note that $|K_t|\leq k$. Set $K_r:=\emptyset$.
        
Define a weighting $w$ of $T$, where $w_t:=|B_t\setminus K_t|$ for each $t\in V(T)$. Thus, $\sum_{t\in V(T)}w_t=n$. By \cref{treeDeletions}, there is a set $Z$ with $|Z|\leq \ceil{n/\delta}-1\leq n/\delta$ such that each connected component of $T-Z$ has weight (under $w$) at most $\delta$. Let $X:=\bigcup_{z\in Z}K_z$. Note that $|X|\leq k|Z|\leq kn/\delta$. Let $Z':=Z\cup \{r\}$. Observe that $|Z'|\leq |Z|+1\leq n/\delta + 1\leq 2n/\delta$ since $\delta\leq n$.

Let $F$ be the forest obtained from $T$ by deleting the edge $p(z)z$ for each $z\in Z'\setminus \{r\}$; see \cref{FigBreakup}. Consider any connected component $T'$ of $F$. Note that $T'$ is a subtree of $T$, the induced root of $T'$ is a vertex $z\in Z'$ (since $r\in Z'$), and that $V(T')\cap Z'=\{z\}$. For each $z\in Z'$, let $T_z$ be the component of $F$ containing $z$. Let $G_z:=G[\bigcup_{t\in V(T_z)}B_t\setminus K_t]$, and let $n_z:=|V(G_z)|$. Observe that $V(T)=\bigcup_{z\in Z'}V(T_z)$, and that $V(T_z)$ is disjoint from $V(T_{z'})$ whenever $z,z'\in Z'$ are distinct. Thus, $\bigcup_{z\in Z'}V(G_z)=V(G)$, and $V(G_z)$ is disjoint from $V(G_{z'})$ whenever $z,z'\in Z'$ are distinct. Therefore, $\sum_{z\in Z'}n_z=n$.
        \begin{figure}[!ht]
        \begin{center}
        \begin{minipage}{0.85\textwidth}
        \centering
        \includegraphics[width=\textwidth]{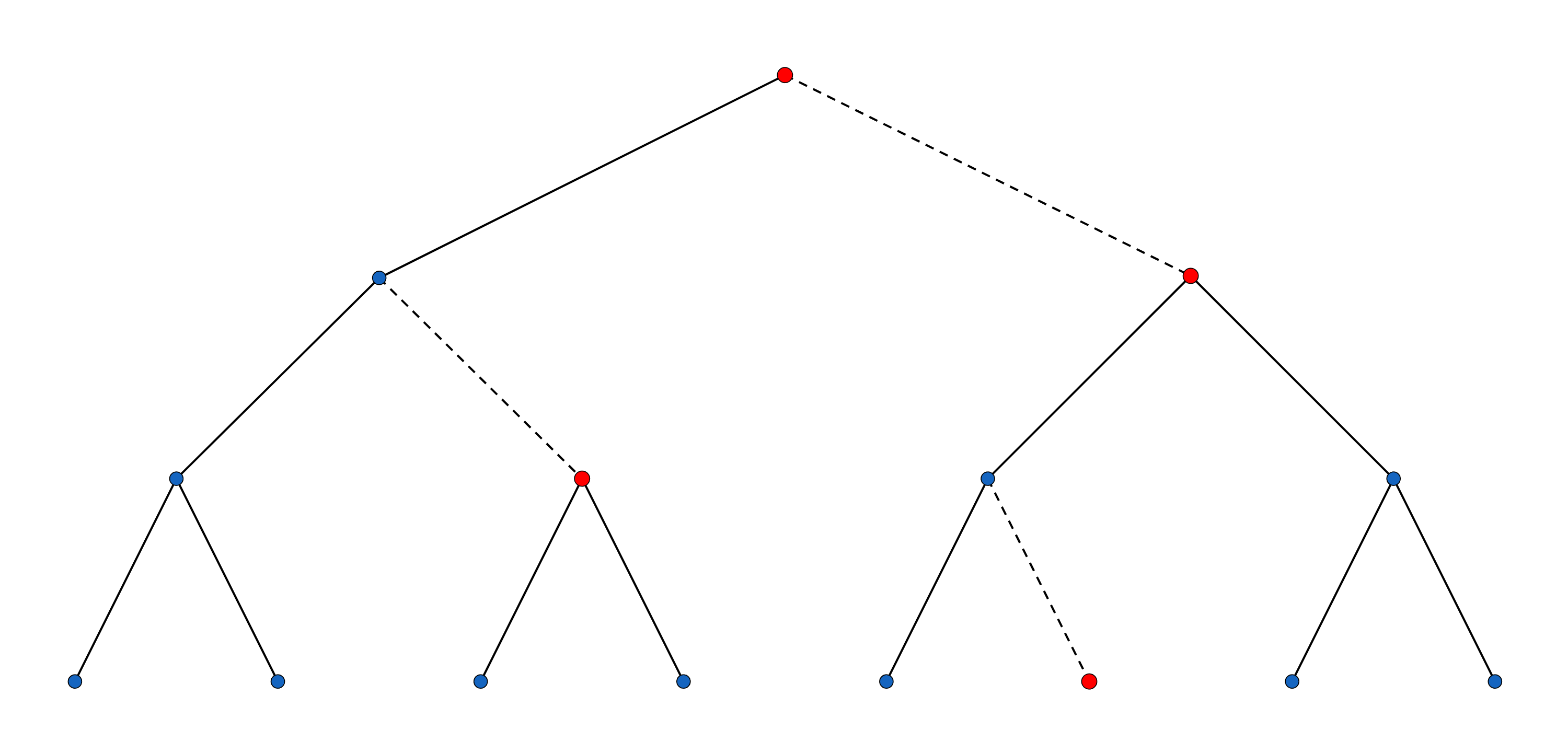}
        \caption{\label{FigBreakup}A diagram of the tree $T$, with the vertices in $Z'$ coloured red and the edges removed to make $F$ represented as dashed lines.}
        \end{minipage}
        \end{center}
        \end{figure}


For each $z\in Z'$, let $\scr{T}_z$ denote the set of connected components of $T_z-z$. Let $S$ be the star with centre $z$ and leaves $\scr{T}_z$. Let $B'_z:=B_z\setminus K_z$, and for each $Q\in \scr{T}_z$, let $B'_Q:=\bigcup_{t\in V(Q)}B_t \setminus K_z$. It follows that $(B'_s:s\in V(S))$ is a star-decomposition of $G_z$ of adhesion at most $k$, and $\UWT{G}{B_z}\in \Apex{\GG}{a}$ by the definition of $(B_t:t\in V(T))$. Since $\UWT{G_z}{B_z'}\subseteq \UWT{G}{B_z}$ and since $\Apex{\GG}{a}$ is monotone, $\UWT{G_z}{B_z'}\in \Apex{\GG}{a}$. By the definition of $Z$ and $w$, for each $Q\in \scr{T}_z$, $|B'_Q\setminus B'_z|\leq \delta$.

Thus, by \cref{starDecomp}, 
$G_z$ has $(f(kn_z)/\delta+a,\max(2\delta\,g(kn_z),\delta))$-bandwidth. So there exists $X_z\subseteq V(G_z)$ such that $|X_z|\leq f(kn_z)/\delta+a$ and $\bw(G_z-X_z) \leq \max(2\delta\,g(kn_z),\delta)=2\delta\,g(kn_z)$. 

        
Let $X':=X\cup \bigcup_{z\in Z'}X_z$. 
Since $f$ is superadditive
\[\sum_{z\in Z'}|X_z|\leq \sum_{z\in Z'}(f(kn_z)/\delta+a)\leq f(k\sum_{z\in Z'}n_z)/\delta+a|Z'| \enspace .\]
Since $\sum_{z\in Z'}n_z=n$ and $|Z'|\leq 2n/\delta$, 
\[\sum_{z\in Z'}|X_z| \leq (f(kn)+2an)/\delta\enspace .\]
Since $|X|\leq kn/\delta$, 
\[|X'|\leq (f(kn)+(2a+k)n)/\delta\enspace .\]
Now, each connected component $C$ of $G-X'$ is contained in $G_z-X_z$ for some $z\in Z'$. Thus, $\bw(C) \leq 2\delta\,g(kn_z)\leq 2\delta\,g(kn)$ since $g$ is non-decreasing. Since this is true for each connected component of $G-X'$, 
$G-X'$ itself has bandwidth at most $2\delta\,g(kn)$, as desired. 
 \end{proof}

We now complete the proof of the main result of this section. 

    \begin{proof}[Proof of \cref{minor-flex}]
    Let $b,k$ be from \cref{GMST}, which depend only on $h$. Let $f,g$ be from \cref{rtw-flex}. So $f\in O(bn\log n)$ and $g\in O(\log^3 n)$. We may take $f$ to be superadditive, and $g$ to be non-decreasing and mapping to $\R_{\geq 1}$. Thus, \cref{mainLem} is applicable. Finally, observe that $n\mapsto f(kn)+(2\max(h-5,0)+k)n \in O_h(n\log n)$ and $n\mapsto 2g(kn) \in O_h(\log^3 n)$.
    \end{proof}

\section{Graphs on Surfaces and With Crossings}
\label{genus_section}


This section proves results for graphs embeddable in fixed surfaces, even allowing a bounded number of crossings per edge. Analogous results with larger dependence on the parameters follow from our previous results for graphs of bounded row treewidth, since each of the classes studied in this section have bounded row treewidth \cite{dujmovic.joret.ea:planar,dujmovic.morin.ea:graph,DHSW24,ueckerdt.wood.ea:improved,dhhw22,HW24}.



The \defin{Euler genus} of a surface obtained from a sphere by adding $h$ handles and $c$ crosscaps is $2h+c$. The \defin{Euler genus} of a graph $G$ is the minimum Euler genus of a surface in which $G$ embeds without crossings. The following result  of \citet{eppstein:dynamic} is useful.\footnote{\cref{eppstein_planarizer} follows from Lemma~5.1 and the proof~of~Theorem~5.1 in \cite{eppstein:dynamic}.}

\begin{thm}[\cite{eppstein:dynamic}]
\label{eppstein_planarizer}
  Every $n$-vertex graph $G$ with Euler genus~$g$ has a set of $X$ of $O(\sqrt{gn})$ vertices such that $G-X$ is planar.
\end{thm}



\begin{lem}
\label{genus_D}
For every $\delta\in\R_{\geq1}$ and $g,n\in\NN$, every $n$-vertex graph $G$ of Euler genus $g$ has a set $X$ of $O(\sqrt{gn}+(n\log n)/\delta)$ vertices such that $G-X$ has bandwidth at most $O(\delta\log^3 n)$.
\end{lem}

\begin{proof}
By \cref{eppstein_planarizer}, $G$ has a set $X_0$ of $O(\sqrt{gn})$ vertices such that $G-X_0$ is planar.  By \cref{planar-flex}, $G-X_0$ has a set $X_1$ of $O((n\log n)/\delta)$ vertices such that $G-(X_0\cup X_1)$ has bandwidth at most $O(\delta \log^3 n)$. The result follows by taking $X:=X_0\cup X_1$.
\end{proof}

The next theorem follows from \cref{UniversalBlowup,genus_D} with multiplier  $\delta=\sqrt{n}/\log n$.

\begin{thm}\label{genus-fanblowup}
For any $g,n\in\NN$, there exists a $O(\sqrt{gn}+\sqrt{n}\log^2n)$-blowup of a fan that contains every $n$-vertex graph of Euler genus at most $g$.
\end{thm}

Our results also generalize for graphs that can be drawn with a bounded number of crossings on each edge. A graph $G$ is \defin{$k$-planar} if it has a drawing in the plane in which each edge participates in at most $k$ crossings, and no three edges cross at the same point. This topic is important in the graph drawing literature; see \cite{KLM17} for a survey just on the $k=1$ case. We use the following bound on the edge density of $k$-planar graphs by \citet{PachToth97}, which is readily proved using the Crossing Lemma~\cite{ajtai.chvatal.ea:crossing_free}.

\begin{thm}[\citep{PachToth97}]
\label{k_planar_density}
For any $k,n\in\NN_1$, every $n$-vertex $k$-planar graph has $O(k^{1/2} n)$ edges.
\end{thm}


\begin{thm}\label{k_planar_D}
For any $k\in\NN_1$, the class of $k$-planar graphs is $( O(k^{3/2}n\log n),O(k\log^3 n) )$-bandwidth-flexible. 
\end{thm}

\begin{proof}
Let $\delta\in\R_{\geq 1}$ be an arbitrary multiplier. We may assume that $n>k\delta$. Let $G$ be a $k$-planar graph. Fix a drawing of $G$ in which each edge is in at most $k$ crossings. Let $G'$ be the planar graph obtained from $G$ by replacing each crossing by a dummy vertex with degree 4, where the portion of an edge of $G$ between two consecutive crossings or vertices becomes an edge in $G'$. By \cref{k_planar_density}, $|E(G)| \in O(k^{1/2} n)$, so the number of dummy vertices introduced this way is $O(k^{3/2} n)$.  Thus $G'$ has $n'\in O(k^{3/2} n)$ vertices. 
By \cref{sparsifier_baker}, $G'$ has a set $X'$ of $O((n'\log n)/\delta)$ vertices such that $G'-X'$ has local density at most $\delta$.  By \cref{rao_bandwidth_vs_density}, $G'-X'$ has bandwidth $O(\delta \log^3 n)$. Let $v_1,\ldots,v_{n'-|X'|}$ be an ordering of $V(G'-X')$ with bandwidth $b:=\bw_{G'}(v_1,\ldots,v_{n'-|X'|})\in O(\delta \log^3 n)$. 

Define the set $X$ by starting with $X:=X'$ and then replacing each (dummy) vertex $x$ in $X\setminus V(G)$ with the endpoints of the two edges of $G$ that cross at $x$.  Then $|X|\le 4|X'|=O((n'\log n)/\delta ) = O((k^{3/2}n\log n)/\delta )$. Now consider any edge $v_i v_j$ of $G-X$. Since $v_i\not\in X$ and $v_j\not\in X$, $G'-X'$ contains a path from $v_i$ to $v_j$ of length at most $k+1$.  Therefore $|i-j|\le (k+1)b\in O(k\delta  \log^3 n)$, and $\bw(G-X)\in O(k\delta \log^3 n)$.
\end{proof}

The next result follows from \cref{UniversalBlowup,k_planar_D} with multiplier $k^{1/4}n^{1/2}/\log n$. 

\begin{cor}\label{k-planar-fanblowup}
    For any $k,n\in\NN_1$ there exists a $O(k^{5/4}\sqrt{n}\log^2n)$-blowup of a fan that contains every $n$-vertex $k$-planar graph.
\end{cor}

The following definition generalizes graphs of Euler genus $g$ and $k$-planar graphs. A graph $G$ is \defin{$(g,k)$-planar} if it has a drawing in a surface of Euler genus at most $g$ in which each edge is in at most $k$ crossings, and no three edges cross at the same point. 
To prove our results, we need a bound on the edge density like \cref{k_planar_density}.  To establish this, we use the following result of \citet{SSSV96}, which generalizes the Crossing Lemma to drawings of graphs on surfaces. 
For a graph $G$ and any $g\in\N$, let \defin{$\crr_g(G)$} denote the minimum number of crossings in any drawing of $G$ in any surface of Euler genus $g$ (with no three edges crossing at a single point).

\begin{lem}[\cite{SSSV96}]\label{shahroki}
  For every $g,n,m\in\N$ with $m\ge 8n$, for every graph $G$ with $n$ vertices and $m$ edges, 
  \[
      \crr_g(G) \ge 
      \begin{cases}
        \Omega(m^3/n^2) & \text{if $0 \le g< n^2/m$} \\
        \Omega(m^2/g) & \text{if $n^2/m \le g\le m/64$}. \\
      \end{cases}
  \]
\end{lem}

\begin{thm}
\label{gkPlanar}
For every $\delta\in\R_{\geq 1}$, $g\in\N$ and $n,k\in\N_1$, 
every $n$-vertex $(g,k)$-planar graph $G$ has a set $X$ of $O(k^{3/4}g^{1/2}n^{1/2} + (k^{3/2}n\log n)/\delta)$ vertices such that $G-X$ has bandwidth $O(\delta k\log^3 n)$.
\end{thm}

\begin{proof}
Let $m:= |E(G)|$. We first show that $k^{3/4}g^{1/2}n^{1/2} \in \Omega(n)$ or that $m\in O(k^{1/2} n)$.  In the former case, taking $X:=V(G)$ trivially satisfies the requirements of the lemma.  We then deal with the latter case using a combination of the techniques used to prove \cref{genus-fanblowup,k-planar-fanblowup}.  

We may assume that $m\ge 64n$ since otherwise $m\in O(k^{1/2}n)$.  We may also assume that $k\le n^{2/3}$ and that $g\le n$ since, otherwise $k^{3/4}g^{1/2}n^{1/2}\ge n$.  (Note that these two assumptions imply that $g\le n\le m/64$.)  If $g< n^2/m$ then, by \cref{shahroki}, the $(g,k)$-planar embedding of $G$ has $\Omega(m^3/n^2)$ crossings.  Since each edge of $G$ accounts for at most $k$ of these crossings,  $km \ge \Omega(m^3/n^2)$, from which we can deduce that $m\in O(k^{1/2} n)$.   If $g\ge n^2/m$ then, by \cref{shahroki}, $G$ has $\Omega(m^2/g)$ crossings and, by the same reasoning, we deduce that $m\in O(kg)\subseteq O(kn)$, since $g\le n$.  Since $g\ge n^2/m$,
   \[
        k^{3/2}g \ge \frac{k^{3/2}n^2}{m} \ge \Omega\left(\frac{k^{3/2}n^2}{kn}\right) = \Omega(k^{1/2}n) \ge \Omega(n) \enspace .
   \]
   Multiplying by $n$ and taking square roots yields $k^{3/4}g^{1/2}n^{1/2} \ge\Omega(n)$.  
   
   We are now left only with the case in which $m\in O(k^{1/2} n)$.   Let $G'$ be the graph of Euler genus at most $g$ obtained by adding a dummy vertex at each crossing in $G$. Then $n':=|V(G')| \le n + km/2 \in O(k^{3/2} n)$ and $\log n' = O(\log n)$, since $k\le n^{2/3}$.
Now apply \cref{eppstein_planarizer} to obtain $X_1\subseteq V(G')$ such that $G'-X_1$ is planar and 
\[
   |X_1| \le \sqrt{gn'} = O(g^{1/2}k^{3/4}n^{1/2}) \enspace .
\]
Now apply \cref{planar-flex} to $G'-X_1$ to obtain a set $X_2\subseteq V(G'-X_1)$ such that 
\[
   |X_2| \le O(n'\log n'/\delta) = O((k^{3/2}n\log n)/\delta) 
\]
and $G'-(X_1\cup X_2)$ has local density at most $D$. Let $X$ be obtained from $X_1\cup X_2$ by replacing each dummy vertex $x$ with the endpoints of the two edges of $G$ that cross at $x$. Then $|X|\le 4|X_1\cup X_2|\in O(k^{3/4}g^{1/2}n^{1/2} + (k^{3/2}n\log n)/\delta)$. By \cref{rao_bandwidth_vs_density}, the bandwidth of $G'-X$ is $O(\delta\log^3 n)$.  Since each edge of $G$ corresponds to a path of length at most $k+1$ in $G'$, this implies that $\bw(G-X)\in O(\delta k\log^3 n)$. 
\end{proof}

The next result, which generalizes \cref{planar-fanblowup,genus-fanblowup,k-planar-fanblowup}, 
follows from \cref{UniversalBlowup,gkPlanar} with multiplier $\delta=k^{1/4}n^{1/2}/\log n$.

\begin{cor}
\label{gk-planar-fan-blowup}
For any $g,k,n\in\NN$ there exists a 
$O( k^{3/4}g^{1/2}n^{1/2} + k^{5/4}n^{1/2}\log^2 n)$-blowup of a fan that contains every $n$-vertex $(g,k)$-planar graph.
\end{cor}

\section*{Acknowledgments} 

This research was initiated at the \emph{Eleventh Annual Workshop on Geometry and Graphs}\linebreak (WoGaG~2024) held at the Bellairs Research Institute of McGill University, March 8--15, 2024. The authors are grateful to the workshop organizers and other participants for providing a working environment that is simultaneously stimulating and comfortable. A preliminary version of this paper (without the results for $K_h$-minor-free graphs) was presented at the \emph{ACM-SIAM Symposium on Discrete Algorithms} (SODA 2025). The authors thank the anonymous referees of the SODA submission and of the journal submission for their careful reading and helpful comments.


\appendix 
\section*{Appendix}
\section{Proof of \texorpdfstring{\cref{reciprocal_sum}}{Lemma?}}
\label{reciprocal_sum_section}

\reciprocalsum*

\begin{proof}[Proof of \cref{reciprocal_sum}]
  First we claim that, for any $x\in S$,
  \begin{equation}
    \sum_{y\in S\setminus\{x\}} \frac{1}{d(x,y)} \le \sum_{i=1}^{n-1}\frac{1}{i/D} = DH_{n-1} < DH_n \enspace . \label{reciprocal_crux}
  \end{equation}
  

To see this, let $d_1 \leq \dotsb \leq d_{n - 1}$ denote the distances of the elements in $S \setminus \{x\}$ from $x$. We claim that $d_i \geq i/D$ for all $i \in \{1, \dotsc, n - 1\}$. If not, then there is some $i$ with $d_i < i/D$. Then $B_{(S, d)}(x, d_i)$ has radius $r  \coloneqq d_i < i/D$ and size $i + 1$ (since it contains $x$), contradicting the fact that $(S, d)$ has local density at most $D$. Therefore, $\sum_{y \in S\setminus\{x\}} \frac{1}{d(x,y)} = \sum_{i = 1}^{n - 1} 1/d_i \leq \sum_{i = 1}^{n - 1} \frac{1}{i/D}$ and so \cref{reciprocal_crux} holds.


  For a set $K$, let $\Pi(K)$ denote the set of all permutations $\pi:\{1,\ldots,k\}\to K$.
  \citet[Lemma~17]{feige:approximating} shows that, for any $k$-element subset $K$ of $S$,
  \[
    \frac{2^{k-1}}{\tvol(K)} \le \sum_{\pi\in\Pi(K)}\frac{1}{\prod_{i=1}^{k-1}d(\pi(i),\pi(i+1))} \enspace .
  \]
  Therefore, to prove the lemma it is sufficient to show that
  \begin{equation}
    \sum_{K\in\binom{S}{k}}\sum_{\pi\in\Pi(K)}\frac{1}{\prod_{i=1}^{k-1}d(\pi(i),\pi(i+1))} \le n(DH_n)^{k-1} \enspace .
    \label{volume_sum}
  \end{equation}
  The proof is by induction on $k$.  When $k=1$, the outer sum in \cref{volume_sum} has $\binom{n}{1}=n$ terms, each inner sum has $1!=1$ terms, and the denominator in each term is an empty product whose value is $1$, by convention.  Therefore, for $k=1$, \cref{volume_sum} asserts that $n \le n$, which is certainly true.  Now assume that \cref{volume_sum} holds for $k-1$.  Then
  \begin{align*}
    & \quad \sum_{K\in\binom{S}{k}}\sum_{\pi\in\Pi(K)}\frac{1}{\prod_{i=1}^{k-1}d(\pi(i),\pi(i+1))} \\
    & = \sum_{K'\in\binom{S}{k-1}}\sum_{\pi\in\Pi(K')}\sum_{y\in S\setminus K'}\frac{1}{\prod_{i=1}^{k-2}d(\pi(i),\pi(i+1))}\cdot\frac{1}{d(\pi(k-1),y)} \\
    & = \sum_{K'\in\binom{S}{k-1}}\sum_{\pi\in\Pi(K')}\frac{1}{\prod_{i=1}^{k-2}d(\pi(i),\pi(i+1))}\cdot\sum_{y\in S\setminus K'}\frac{1}{d(\pi(k-1),y)} \\
    & \le \sum_{K'\in\binom{S}{k-1}}\sum_{\pi\in\Pi(K')}\frac{1}{\prod_{i=1}^{k-2}d(\pi(i),\pi(i+1))}\cdot DH_n & \text{(by \cref{reciprocal_crux})}\\
    & \le n(DH_n)^{k-2}DH_n & \text{(by induction)}\\
    & = n(DH_n)^{k-1} \enspace . & & \qedhere
  \end{align*}
\end{proof}

\section{Proof of \texorpdfstring{\cref{volume_density_bandwidth}}{Lemma?}}
\label{volume_density_bandwidth_section}

\volumedensitybandwidth*

\begin{proof}[Proof of \cref{volume_density_bandwidth}]
  Let $r$ be a random unit vector in $\R^L$ and for each $v\in S$, let $\mathdefin{h(v)}:=\langle r,\phi(v)\rangle$ be the inner product of $r$ and $\phi(v)$.  We will order the elements of $S$ as $v_1,\ldots,v_n$ so that $h(v_1)\le \cdots \le h(v_n)$.  To prove an upper bound on $\bw(S,d)$, it suffices to show an upper bound that holds with positive probability on the maximum, over all $vw$ with $d(v,w)\le 1$, of the number of vertices $x$ such that $h(v)\le h(x)\le h(w)$.

  Consider some pair $v,w\in S$ with $d(v,w)\le 1$. Since $\phi$ is a contraction, $d_2(\phi(v),\phi(w))\le 1$.  By \cite[Proposition~7]{feige:approximating}, $\Pr(|h(v)-h(w)|\ge \sqrt{4a\ln n/L}) \le n^{-a}$, for any $a\ge 1/(4\ln n)$. Therefore, with probability at least $1-n^{-a+2}$, $|h(v)-h(w)|\le \sqrt{4a\ln n/L}$ for each pair $v,w\in S$ with $d(v,w)\le 1$.

  Let  $K:=\{v_1,\ldots,v_k\}$ be a $k$-element subset of $S$. First observe that $\evol(\phi(K)) \le \Delta^k$, since $\evol(\phi(K))\le\prod_{i=2}^k d_2(\phi(v_{i-1}),\phi(v_i))\le \prod_{i=2}^k d(v_{i-1},v_i)\le \Delta^{k-1}$.  In particular, \linebreak $\log\evol(\phi(K))\le k\log \Delta$.

  Define $\mathdefin{\ell(K)}:=\max_{v\in K}h(v)-\min_{v\in K} h(v)$.  By \cite[Theorem~9]{feige:approximating} there exists a universal constant $\beta$ such that, for any $c>0$,
  \[
      \Pr(\ell(K) \le c)
        < \frac{(\beta L)^{k/2}c^k\max\{1,\log(\evol(\phi(K))\}}{k^k\evol(\phi(K))}
        \le \frac{(\beta L)^{k/2}c^kk\log\Delta}{k^k\evol(\phi(K))} \enspace .
  \]
  In particular,
  \begin{align*}
    \Pr(\ell(K) \le \sqrt{4a\ln n/L})
      & <
      \frac{(4\beta a\ln n)^{k/2}\, k\log\Delta}{k^k\evol(\phi(K))} \\
      & \le \frac{(4\beta a\ln n)^{k/2}\eta^{k-1}\, k\log\Delta}{k^k\ivol(\phi(K))} \\
      & \le \frac{(4\beta a\ln n)^{k/2}\eta^{k-1}(k-1)!2^{(k-2)/2}\, k\log\Delta}{k^k\tvol_{d}(K)} \\
      & \le \frac{(4\beta a\ln n)^{k/2}\eta^{k-1}2^{(k-2)/2}\,k\log\Delta}{\tvol_{d}(K)} \\
      & \le \frac{\left((8\beta a\ln n)^{1/2}\eta\right)^{k}\,k\log\Delta}{\tvol_{d}(K)} \enspace .
  \end{align*}
  Say that $K\in\binom{S}{k}$ is \defin{bad} if $\ell(K) \le \sqrt{4a\ln n/L}$. By \cref{reciprocal_sum}, the expected number of bad sets is
  \begin{align}
    \sum_{K\in \binom{S}{k}} \Pr(\text{$K$ is bad})
    & \le \sum_{K\in \binom{S}{k}} \frac{\left((8\beta a\ln n)^{1/2}\eta\right)^{k}\, k\log\Delta}{\tvol_{d}(K)} \notag \\
    & \le \left((8\beta a\ln n)^{1/2}\eta DH_n\right)^{k}\, nk\log\Delta
    \enspace .
     \label{bad_expectation}
  \end{align}
  Let $B$ be a maximum cardinality subset of $S$ with $\ell(B)<\sqrt{4a\ln n/L}$.  The vertices in $B$ form $\binom{|B|}{k}$ bad sets. Therefore, by Markov's Inequality, the probability that $\binom{|B|}{k}$ exceeds \eqref{bad_expectation} by a factor of at least $2$ is at most $1/2$.  Therefore, with probability at least $1/2$, $\binom{|B|}{k}\le 2\left((8\beta a\ln n)^{1/2}\eta DH_n\right)^{k}\, nk\log\Delta$, which implies that $|B|\in O((nk\log\Delta)^{1/k}\,Dk\eta\log^{3/2} n)$ with probability at least $1/2$.\footnote{Very roughly, $\binom{|B|}{k}$ is approximated by $(|B|/k)^k$.}  Therefore, with probability at least $1/2-n^{-a+2}$, $\bw_d(x_1,\ldots,x_n)\in O((nk\log\Delta)^{1/k}\,Dk\eta\log^{3/2} n)$.
\end{proof}


\section{Proof of \texorpdfstring{\cref{volume_preserver}}{Claim?}}
\label{volume_preserver_proof}

\volumepreserverClaim*

\begin{proof}
  The following argument is due to \citet[pages~529--530]{feige:approximating}.  Let $K$ be a set of $k$ vertices of $(H\boxtimes P)-X$. Let $T$ be a minimum spanning tree of the complete graph on $K$ where the weight of an edge $xy$ is $d^*(x,y)$.  Let $x_1,\ldots,x_k$ be an ordering of the vertices in $K$ and $T_0,\ldots,T_k$ be a sequence of trees such that $T_{p}$ is a minimum spanning tree of $x_1,\ldots,x_{p}$ that contains $T_{p-1}$ as a subgraph, for each $p\in\{1,\ldots,k\}$.  That such an ordering and sequence of trees exist follows from the correctness of Prim's Algorithm. For each $p\in\{2,\ldots,k\}$, let $h_p:=d^*(x_p,\{x_1,\ldots,x_{p-1}\})$ be the cost of the unique edge in $E(T_p)\setminus E(T_{p-1})$.  Observe that $\prod_{p=2}^k h_p = \tvol_{d^*}(K)$.

  For each $p\in\{2,\ldots,k\}$, let $B_{p-1}:=\left\{\sum_{i=1}^{p-1}\lambda_i\phi(v_i):(\lambda_1,\ldots,\lambda_{p-1})\in\Gamma_{p-1}\right\}$ be the subspace of $\R^L$ spanned by $\phi(v_1),\ldots,\phi(v_{p-1})$.  Then\footnote{This is the $(p-1)$-dimensional generalization of the formula $a:=bh/2$ for the area $a$ of a triangle $v_1,v_2,v_3$ with base length $b=\evol(\{v_1,v_2\})$ and height $h=d_2(v_3,B_2)$, where $B_2$ is the line containing $v_1$ and $v_2$.}
  \[
    \evol(\{v_1,\ldots,v_p\})=\frac{\evol(\{v_1,\ldots,v_{p-1}\})\cdot d_2(v_p,B_{p-1})}{p-1} \enspace .
  \]
  Observe that each coordinate $\phi_{i,j}(v_p)$ of $\phi(v_p)$ is at most $2(n-1)$, since $\phi_{i,j}(v_p)=\alpha_{i,j}\cdot d_{H\boxtimes P}(v_p, \overline{C_{i,j}})\le 2(n-1)$. Therefore, $\phi(v_p)$ is contained in a ball $B$ of radius $2(n-1)\sqrt{L}$ around the origin. \citet{feige:approximating} uses these two facts to show $B_{p-1}\cap B$ can be covered by $\Theta(n^{2k})$ balls, each of radius $h_p\zeta$, such that, if $\phi(v_p)$ is not contained in any of these balls, then $d_2(\phi(v_p),B_{p-1})\ge 2h_p\zeta/3$.  When this happens,
  \[
    \evol(\{\phi(v_1),\ldots,\phi(v_p)\})\ge \frac{(2h_p\zeta/3)\cdot\evol(\{\phi(v_1),\ldots,\phi(v_{p-1})\})}{p-1} \enspace .
  \]
  By \cref{crux}, the probability that $\phi(v_p)$ is not contained in any of these balls is at least $1-O(n^{-k})$. By the union bound, the probability that this occurs for each $p\in\{2,\ldots,k\}$ is at least $1-O(kn^{-k})$.  Therefore, with probability at least $1-O(kn^{-k})$,
  \[
    \evol(\phi(K)) \ge \prod_{p=2}^{k}\frac{2h_p\zeta/3}{p-1} = \frac{\tvol_{d^*}(K)\cdot (2\zeta/3)^{k-1}}{(k-1)!} \enspace . \qedhere
  \]
\end{proof}

\newpage
\begin{aicauthors}
\begin{authorinfo}[marc]
Marc Distel\\
School of Mathematics\\
Monash University\\
Melbourne, Australia\\
marc\imagedot{}distel\imageat{}monash\imagedot{}edu
\end{authorinfo}
\begin{authorinfo}[vida]
Vida Dujmovi{\'c}\\
School of Computer Science and Electrical Engineering\\
University of Ottawa\\
Ottawa, Canada\\ 
vida\imagedot{}dujmovic\imageat{}uottawa\imagedot{}ca
\end{authorinfo}
\begin{authorinfo}[gwen]
Gwena\"el Joret\\
D\'epartement d'Informatique\\
Universit\'e libre de Bruxelles\\
Brussels, Belgium\\
gwenael\imagedot{}joret\imageat{}ulb\imagedot{}be\\
\end{authorinfo}
\begin{authorinfo}[piotrek]
Piotr Micek\\
Department of Theoretical Computer Science\\
Jagiellonian University\\
Kraków, Poland\\ 
piotr\imagedot{}micek\imageat{}uj\imagedot{}edu\imagedot{}pl
\end{authorinfo}
\begin{authorinfo}[pat]
Pat Morin\\
School of Computer Science\\
Carleton University\\
Ottawa, Canada\\ 
morin\imageat{}scs\imagedot{}carleton\imagedot{}ca
\end{authorinfo}
\begin{authorinfo}[david]
David R. Wood\\
School of Mathematics\\
Monash University\\
Melbourne, Australia\\
david\imagedot{}wood\imageat{}monash\imagedot{}edu
\end{authorinfo}
\end{aicauthors}
\end{document}